\documentclass[10pt,twoside,reqno]{amsart} 
\usepackage{amsmath, amsfonts, amsthm, amssymb}
\usepackage[french,english]{babel}
\usepackage[latin1]{inputenc}
\usepackage{amsmath}
\usepackage{amsfonts}
\usepackage{amsthm}
\usepackage[]{graphics}
\usepackage{color}
\usepackage[T1]{fontenc}
 \usepackage{setspace}
\usepackage{soul}
\usepackage{algorithm}
\usepackage{algorithmic}

\usepackage[T1]{fontenc}
\usepackage[latin1]{inputenc}
\usepackage{amsmath,latexsym}
\usepackage{here}
\usepackage{graphicx}
\usepackage{amssymb}
\usepackage{caption}
\usepackage{subcaption}

\definecolor{red}{rgb}{0.82,0.00,0.00}  

\numberwithin{equation}{section}
\input epsf
 

\newcommand{\norm}[1]{\parallel #1 \parallel}
 \newcommand{\dd}{{\rm div}}
\newcommand{\ds}{\displaystyle}

\newcommand{\Om}{{\Omega}}

\def\R{{\rm I\hspace{-0.50ex}R} }
\def\P{{\rm I\hspace{-0.50ex}P} }

\def\w{{\bf w}}
\def\f{{\bf f}}
\def\v{{ \bf v}}
\def\u{{\bf u}}

\def\0{{\bf 0}}

\def\x{{\bf{x}}}

\def\s{s}
\def\n{{\bf n}}
\def\div{\operatorname{div}}

\newtheorem{lem}{Lemma}[section]
\newtheorem{thm}[lem]{Theorem}
\newtheorem{rmq}[lem]{Remark}

\newtheorem{Rem}[lem]{Remark}

\newtheorem{propri}[lem]{Property}
\renewcommand{\epsilon}{\varepsilon}
\def\twoplot[#1]#2#3#4#5{
\begin{figure}[hbt]
\begin{multicols}{2}
\begin{center}
    \includegraphics*[#1]{#2}
    \caption{\label{#2} #4}
\end{center}
\begin{center}
    \includegraphics*[#1]{#3}
    \caption{\label{#3} #5}
\end{center}
\end{multicols}
\end{figure}
}

\begin{document}

\bibliographystyle{plain}

\title[{\it A posteriori} error estimates for Darcy-Forchheimer's problem]{ {\it A posteriori} error estimates for Darcy-Forchheimer's problem}
\author[ SAYAH  ]{Toni Sayah$^{\dagger}$}
\author[G. Semaan]{Georges Semaan$^\dagger$}
\author[F. Triki]{Faouzi Triki$^\ddagger$}
\thanks{ \today.
\newline
$^{\dagger}$  Laboratoire de "Math\'ematiques et Applications", Unit\'e de Recherche "Math\'ematoqies et Mod\'elisation", CAR, Facult\'e des sciences de Beyrouth, Universit\'e Saint-Joseph de Beyrouth, Liban.
\newline
Emails :toni.sayah@usj.edu.lb, georges.semaan2@net.usj.edu.lb,\\
$^\ddagger$ Laboratoire Jean Kuntzmann, UMR CNRS 5224, Universit\'e Grenoble-Alpes, 700 Avenue Centrale, 38401 Saint-Martin-d'H\`eres, France. E-mails: faouzi.triki@univ-grenoble-alpes.fr. FT was supported by the grant ANR-17-CE40-0029 of the French National Research Agency ANR (project MultiOnde)}


\begin{abstract}
\noindent This work deals with the \emph{a posteriori} error estimates for the Darcy-Forchheimer problem. We introduce the corresponding variational formulation and discretize  it  by using the finite-element method. {\it A posteriori} error estimate with two types of computable error indicators is showed. The first one is linked to the linearization and the second one to the discretization. Finally, numerical computations are performed  to show the effectiveness of the error  indicators.
\end{abstract}

\maketitle

\section{Introduction}\label{intro}
\noindent Let $\Omega$ be a bounded open domain of $\R^d$ ($d = 2, 3$) with  Lipschitz-continuous boundary $\Gamma=\partial \Omega$. We consider the Darcy-Forchheimer equation
\begin{equation}\label{E1}
\ds \frac{\mu}{\rho} K^{-1}  \u + \frac{\beta}{\rho} |\u| \u+ \nabla p = \f \;\;\;\;\; \mbox{in}\;\;\; \Omega,
\end{equation}
with the divergence constraint
\begin{equation}\label{E2}
\ds \div \u = b \;\;\;\;\; \mbox{in}\;\;\; \Omega\textcolor{blue}{,}
\end{equation}
and the boundary condition
\begin{equation}\label{E3}
\ds \u \cdot \n = g \;\;\;\;\; \mbox{on}\;\; \partial \Omega,
\end{equation}

where $\u$ represents the velocity, $p$ represents the pressure, $\n$ is the unit exterior normal vector to $\Gamma$, $|.|$ denotes the Euclidean norm and $|\u|^2 = \u \cdot \u$. The parameters $\rho$, $\mu$ and $\beta$ represent the density of the fluid, its viscosity and its dynamic viscosity, respectively. $\beta$ is also referred as Forchheimer number which is a scalar positive constant. $K$ is the permeability tensor, assumed to be uniformly positive definite and bounded such that there exists two positive real numbers $K_m$ and $K_M$ verifying
\begin{equation}\label{KmM}
0 < K_m \, \x \cdot \x \le (K^{-1}(\x)\x) \cdot \x \le K_M \, \x \cdot \x.
\end{equation}
For the compatibility, we suppose that $b$ and $g$ verify the following compatibility condition:
\[
\ds \int_\Omega b(\x) d\x = \ds \int_\Gamma \mbox{g}(s) \, ds.
\]
We denote by Problem $(P)$ the system of equations (\eqref{E1}, \eqref{E2}, \eqref{E3}). \\

Darcy's law describes the creeping flow of Newtonian fluids in porous media. It is represented by Equation \eqref{E1} without the non-linear term $\ds \frac{\beta}{\rho} |\u| \u$ and is valid by experiment under the condition that the creeping velocity is low and the porosity and permeability are small enough by Darcy in $1856$ \cite{Darcy}. Forchheimer \cite{Forchheimer} conducted flow experiments in sandpacks and recognized that when the velocity is higher and the porosity is nonuniform, Darcy's equation can be replaced by Equation \eqref{E1}. A theoretical derivation of Forchheimer's law can be found in \cite{RM}.\\
\noindent Multiple works approximated the Darcy-Forchheimer equation by using finite element methods. Girault and Wheeler \cite{GW} approximated the velocity by piecewise constants and the pressure by Crouzeix-Raviart element. They also proposed and studied an alternating directions iterative method to solve the system of nonlinear equations obtained by finite element discretizaton. Lopez and al. \cite{LMJ} carried out numerical tests of the methods studied in \cite{GW} in order to corroborate the results presented there. In \cite{HH}, the authors proposed and studied a mixed element approximation: the Raviart-Thomas mixed element, Brezzi-Douglas-Marini mixed element. Salas J. et al. \cite{SJ} presented a theoretical study of the mixed finite element space, as proposed in \cite{LMJ}. In \cite{SST}, we considered the Darcy-Forchheimer problem coupled with the convection-diffusion-reaction problem. We established existence of solutions by using a Galerkin method and we proved uniqueness. Then, we introduced and analyzed a numerical scheme based on the finite element method and we derived an optimal {\it a priori} error estimates for the proposed numerical scheme.
In this work, we introduce a numerical iterative scheme, show the corresponding convergence, establish the corresponding {\it a posteriori} error estimates and show corresponding numerical investigations.\\
%

The present work investigates \emph{a posteriori} error estimates for the finite element discretization of  Darcy-Forchheimer problem.These estimates can be used to evaluate the solution errors of the discrete problem without requiring any a priori information on the exact solution. Indeed, the  \emph{a posteriori} analysis controls the overall discretization error of a problem by providing error indicators  that are  easy to compute.  Once these error indicators are constructed, their efficiency can be proven by bounding each indicator by the local error. \emph{A posteriori} analysis was first introduced by I. Babu$\check{s}$ka \cite{Babu78C}, developed by {R.~Verf\"{u}rth} \cite{Verfurth1996}, and has been the object of a large number of publications. {\it A posteriori} error estimations have been studied for several types of partial differential equations such that the Stokes or Navier-Stokes equation  (see for instance \cite{Verfurth1996,bernardi2009,Oden1, sayah1,sayah2,sayah3,sayah5}),  the Maxwell and Lam\'e equations \cite{Monk, Bao}. Many works have been established for the Darcy flow, see for instance \cite{alonso, braess, carstensen, lovadina}. In \cite{chen}, Chen and Wang establish  optimal {\it a poteriori} error estimates for the $H(\dd,\Omega)$ conforming mixed finite element  method applied to   the coupled Darcy-Stokes system  in two dimensions.  For the Darcy equations with pressure dependent viscosity, we refer to \cite{gatica} and the references therein. In \cite{DGHS}, we establish {\it a posteriori} error estimates for Darcy's problem coupled with the heat equation. In this paper, we study the {\it  posteriori} error estimates corresponding to Problem $(P)$ and show corresponding numerical results.

This paper is organised as follow:
\begin{itemize}
\item Section 2 describes the problem and the weak formulation.
\item Section 3 describes the discretization and studies  the iterative scheme.
\item Section 4 is devoted to the {\it a posteriori} error estimates.
\item Section 5 shows numerical investigations.
\end{itemize}
\section{Variational formulation}
In this section, we begin by introducing several definitions and notations in order to write the weak formulation corresponding to Problem $(P)$.\\
Let $\alpha=(\alpha_1,\alpha_2, \dots \alpha_d)$ be a $d$-uplet of non negative
integers, set $|\alpha|=\ds \sum_{i=1}^d \alpha_i$, and define the partial
derivative $\partial^\alpha$ by
$$
\partial^{\alpha}=\ds \frac{\partial^{|\alpha|}}{\partial x_1^{\alpha_1}\partial
x_2^{\alpha_2}\dots\partial x_d^{\alpha_d}}.
$$
Then, for any positive integer $m$ and any number $q\geq 1$, we recall
the classical Sobolev space
\begin{equation}
\label{eq:Wm,p}
W^{m,q}(\Omega)=\{v \in L^q(\Omega);\,\forall\,|\alpha|\leq
m,\;\partial^{\alpha} v \in L^q(\Omega)\},
\end{equation}
equipped with the seminorm
\begin{equation}
\label{eq:semnormWm,p}
|v|_{W^{m,q}(\Omega)}=\big\{\sum_{|\alpha|=m} \int_{\Omega}
|\partial^{\alpha} v|^q \,d\x\,\big\}^{\frac{1}{q}}
\end{equation}
and the norm
\begin{equation}
\label{eq:normWm,p}
\|v\|_{W^{m,q}(\Omega)}=\big\{\sum_{0\leq k\leq m}
|v|_{W^{k,q}(\Omega)}^q \big\}^{\frac{1}{q}}.
\end{equation}
When $q=2$, this space is the Hilbert space $H^m(\Omega)$.
In
particular, the scalar product of $L^2(\Omega)$ is denoted by
$(.,.)$. Furthermore, we recall the following standard space for Darcy's equations:
\[
L^2_0(\Omega) = \{ v\in L^2(\Omega); \ds \int_\Omega v \, d\x =0 \}.
\]
The definitions of these spaces are extended straightforwardly to
vectors, with the same notation, but with the following
modification for the norms in the non-Hilbert case. Let $\v$ be a
vector valued function; we set
\begin{equation}
\label{eq:normLp} \|\v\|_{L^q(\Omega)}= \big(\int _{\Omega}
|\v|^q\,d\x\, \big)^{\frac{1}{q}},
\end{equation}
where $|.|$ denotes the Euclidean vector norm.\\

%
%
%

We introduce the spaces:
\[
\begin{array}{ll}
\medskip
X = L^3(\Omega)^d, \\
M = W^{1, \frac{3}{2}}(\Omega) \cap L^2_0(\Omega),\\
{H = \left\{ \v \in X; \div \v \in L^{3d/(d+3)}(\Omega) \right\},}\\
{ V  = \{ \v\in X ; \forall q \in M , \ds \int_\Omega \nabla q \cdot \v \, d\x= 0  \}.}
\end{array}
\]
{ The space $H$ endowed  with the graph norm is complete. Moreover
$D(\overline{\Omega}) $ is dense in $H$, and $\v \cdot \n$ belongs to $W^{1, 3/2}(\Gamma)$ for all $v \in H$ \cite{Te}. }
We have the following inf-sup condition (see \cite{GW} for the proof)
\begin{equation}\label{inf-supS1}
\underset{q\in M}{\inf} \; \; \underset{\v \in X}{\sup} \; \; \ds \frac{\ds \int_\Omega \nabla q \cdot \v \, d\x}{||\nabla q ||_{L^{3/2}(\Omega)} ||\v||_{L^3(\Omega)}} = 1.
\end{equation}
The velocity $\u$ and pressure $p$ of Problem $(P)$ are considered respectively in $H$ and $M $, {while}  $b$ and $g $ are assumed to be respectively in $L^{3d/(d+3)}(\Omega)$ and $L^{3(d-1)/d}(\Gamma)$ (see \cite{GW} for details). \\

In order to write the variational formulation {associated} to Problem $(P)$, we introduce the mapping $\v \longrightarrow \mathcal{A}(\v)$ defined by:
\begin{equation*}
    \begin{array}{rccl}
  \mathcal{A}:&  L^3(\Omega)^d &\mapsto&L^\frac{3}{2}(\Omega)^d \\
    &\v &\mapsto &\mathcal{A}(\v)=\ds \frac{\mu}{\rho}K^{-1}\v+\frac{\beta}{\rho}|\v|\v. 
    \end{array}
\end{equation*}
We refer to \cite{GW,monot} for the following properties of $\mathcal{A}$.
\begin{propri}\label{monoticity}
We have the following monotonicity properties:
\begin{enumerate}
\item for all $\v \in L^{3}(\Omega)^d$,
\begin{equation}\label{prop1}
\ds \frac{\mu}{\rho} \int_{\Omega}K^{-1} \v \cdot \v\,d\x \geq \ds  \frac{\mu K_m}{\rho}\left\|\v\right\|_{L^2(\Omega)}^2
\end{equation}
\item  
for all $\v,\w \in L^{3}(\Omega)^d$,
\begin{equation}\label{prop2}
\ds \frac{\beta}{\rho} \int_{\Omega} (|\v|\v - |\w|\w )(\v-\w) \,d\x \geq
c_m \left\|\v-\w\right\|_{L^3(\Omega)}^3.
\end{equation}
where $c_m$ is a strictly positive constant. 
\end{enumerate}
\end{propri}
{Recall the following Green's formula \cite{Te}
\begin{equation}\label{green}
\forall q \in M,\;  \forall v \in H,  \quad \ds  \int_\Omega \nabla q \cdot \v \, d\x  =  \ds -\int_\Omega q \div v d\x +\int_\Gamma q \v \cdot \n ds.
\end{equation}
}
{Using the Green formula \eqref{green} one can show that } Problem $(P)$ is equivalent to the following variational formulation: Find $(\u,p)\in X  \times M $ such that
\begin{equation}\label{V1}
\left\{
\begin{array}{ll}
\medskip
\forall \v \in X , \quad \ds  \int_\Omega \mathcal{A}(\u) \cdot \v \, d\x + \int_\Omega \nabla p \cdot \v \, d\x = \ds \int_\Omega \f \cdot \v \, d\x,\\
\forall q\in M , \quad \ds \int_\Omega \nabla q \cdot \u \, d\x = -\int_\Omega b \, q \, d\x + \int_\Gamma g  \, q \, ds.
\end{array}
\right.
\end{equation}
It can be demonstrated that, for all $b \in L^{3d/(d+3)}(\Omega)$ and $g  \in L^{3(d-1)/d}(\Gamma)$, there is only one $\u_l \in {  L^{3}(\Omega)/V }$ such that
\begin{equation} \label{equationul}
\forall q \in M , \, \ds \int_\Omega \nabla q \cdot \u_l \, d\x = -\int_\Omega b \, q \, d\x + \int_\Gamma g  \, q \, ds,
\end{equation}
and
\[
\ds ||\u_l ||_{L^3(\Omega)^d/V } \le C \big( ||b||_{L^{3d/(d+3)}(\Omega)}  + ||g  ||_{L^{3(d-1)/d}(\Gamma)}  \big),
\]
where  $C>0$ is a constant depending only on $\Omega$ and $d$. \\

%
%
%
%

{Using The Banach-Necas-Babu$\check{s}$ka Theorem (see for instance \cite{Babu73}, Theorem A.1 in 
\cite{TrikiYin21}) one can show that \eqref{V1}, and equivalently the problem $(P)$} admits a unique solution $(\u,p) \in X  \times M $ satisfying the relations
\begin{equation}
\begin{array}{ll}\label{bound1u}
\medskip
||\u||_{L^3(\Omega)} \le C \Big(  ||\u_l||_{L^3(\Omega)}^3 + ||\u_l||_{L^2(\Omega)}^2 + ||\f ||_{L^2(\Omega)}^2 \Big)^{1/3}, \\
||\nabla p||_{L^{3/2}(\Omega)} \le C \Big(  ||\u||_{L^{3/2}(\Omega)} +  ||\u ||_{L^3(\Omega)}^2 + ||\f ||_{L^2(\Omega)} \Big).
\end{array}
\end{equation}

{For further details of the weak formulation corresponding to Problem $(P)$ presented in this section, We refer to \cite{GW}. }
\section{Discretization and {\it{a posteriori}} estimation}
In this section, we introduce an iterative numerical scheme  {to solve the weak variational formulation \eqref{V1} of the  Problem $(P)$, and study its convergence.} We establish the  corresponding {\it a posteriori} error estimate {in a second step}. \\

From now on, we assume that $\Omega$ is a polygon when $d=2$ or
polyhedron when $d=3$, so it can be completely meshed.  For
 the space discretization,  we consider a
regular (see Ciarlet~\cite{PGC}) family of triangulations
$( \mathcal{T}_h )_h$ of $\Omega$ which is a set of closed non degenerate
triangles for $d=2$ or tetrahedra for $d=3$, called elements,
satisfying,
\begin{itemize}
\item for each $h$, $\overline{\Omega}$ is the union of all elements of
 $\mathcal{T}_h$;
\item the intersection of two distinct elements of  $\mathcal{T}_h$ is
either empty, a common vertex, or an entire common edge (or face
when $d=3$);
\item the ratio of the diameter   $h_\kappa$  of an element $\kappa$ in  $\mathcal{T}_h$  to
the diameter  $\rho_\kappa$  of its inscribed circle when $d=2$ or ball when $d=3$
is bounded by a constant independent of $h$: there exists a positive constant $\sigma$ independent of $h$ such that,
\begin{equation}
\label{eq:reg}
\ds \max_{\kappa\in \mathcal{T}_{h}} \frac{h_\kappa}{\rho_\kappa} \le \sigma.
\end{equation}
\end{itemize}
As usual, $h$ denotes the maximal diameter
of all elements of $\mathcal{T}_{h}$.   To define the finite element functions,  let $r$ be a non negative integer.  For each $\kappa$ in
$\mathcal{T}_{h}$, we denote by  $\P_r(\kappa)$  the space of restrictions to $\kappa$ of polynomials in $d$ variables and total
degree at most  $r$, with a similar notation on the faces or edges of $\kappa$. For every edge (when $d=2$) or face (when $d=3$) $e$ of the mesh  $\mathcal{T}_h$, we denote by $h_e$ the diameter of $e$.

We shall use the following inverse inequality:  for any numbers $p,q \geq 2$,  for any dimension $d$, and for any non negative integer $r$,  there exist  constants $C_I(p)$ and $C_J(q)$ such that for any polynomial function $v_h$ of degree $r$ on an element $\kappa$ or an edge (when $d = 2$) or face (when $d = 3$) $e$ of the mesh $\mathcal{T}_{h}$,
\begin{equation}
\label{eq:inversin}
\begin{array}{ll}
\medskip
 \|v_h\|_{L^p(\kappa)}\leq  C_I(p)  h_\kappa^{\frac{d}{p}-\frac{d}{2}}\|v_h\|_{L^2(\kappa)}\\
 \mbox{and}\\
 \|v_h\|_{L^q(e)}\leq  C_J(q)  h_e^{\frac{d-1}{q}-\frac{d-1}{2}}\|v_h\|_{L^2(e)}.
 \end{array}
\end{equation}
$C_I$ and $C_J$ depend on the regularity parameter $\sigma$ of \eqref{eq:reg}. In all the rest of the paper and for the simplicity, we will omit the dependency of $C_I$ and $C_J$ with $p$ and $q$.\\

Let $X_h \subset X $ and $M_h\subset M $ be the discrete spaces corresponding to the velocity and the pressure given by:
\begin{equation}
\begin{split}
Z_h =&\{q_h\in C^0(\bar{\Omega}); \, \forall\, \kappa \in \mathcal{T}_h, \;q_{h}|_\kappa \in \P_1 \}, \\
X_h=&\{\v_h\in L^2(\bar{\Omega})^d; \, \forall\, \kappa \in \mathcal{T}_h, \;\v_h|_\kappa \in \P_0^d \}, \\
M_h=& Z_h \cap L^2_0(\Omega).
\end{split}
\end{equation}

They satisfy the following inf-sup condition (see \cite{SJ}):
\begin{equation}
\label{infsuph1}
\forall\, q_h\in M_h,\; \sup_{\v_h\in X_h}\ds \frac{\ds \int_{\Omega} \nabla q_h \cdot \v_h \,  d\x\,}{\|\v_h\|_{X }}\geq \beta_1  \|q_h\|_{M_h},
\end{equation}
where $\beta_1 $ is a positive constant independent of $h$.\\

Problem \eqref{V1} can be discretized as following:
\begin{equation}\label{V1h}
\left\{
\begin{array}{ll}
\medskip
\forall \v_h \in X_h, \quad \ds  \int_\Omega \mathcal{A}(\u_h) \cdot \v_h \, d\x + \int_\Omega \nabla p_h \cdot \v_h \, d\x = \ds \int_\Omega \f \cdot \v_h \, d\x,\\
\forall q_h\in M_h, \quad \ds \int_\Omega \nabla q_h \cdot \u_h \, d\x = -\int_\Omega b q_h \, d\x + \int_\Gamma g  q_h \, ds.
\end{array}
\right.
\end{equation}
It is shown in \cite{SJ} that there exists a unique $\u_{h,l} \in X_h$ such that
\begin{equation}\label{equat11}
\forall q_h \in M_h, \quad \ds \int_\Omega \nabla q_h \cdot \u_{h,l} \, d\x = \ds - \int_\Omega b\, q_h d\x + \int_\Gamma g \, q_h ds,
\end{equation}
and $\u_{h,l}$ satisfies the following bound,
\begin{equation}\label{equat122222}
||\u_{h,l}||_{L^3(\Omega)^d} \le C_l \big( ||b||_{L^{\frac{3d}{3+d}}(\Omega)} + ||g ||_{L^{\frac{3(d-1)}{d}}(\Gamma)} \big).
\end{equation}
It is also shown in \cite{SJ} that Problem \eqref{V1h} admits a unique solution $(\u_h,p_h) \in X_h \times M_h$ satisfying exactly similar bounds as \eqref{bound1u}. We refer to \cite{SJ} for the proof of the following {\it a priori} error estimates:
{\thm \label{thmpriori1}
The solutions $(\u,p)$ of \eqref{V1} and $(\u_h,p_h)$ of $\eqref{V1h}$ verify the following {\it a priori} error: \\
If $(\u,p)\in W^{1,4}(\Omega)^d \times W^{2,3/2}(\Omega)$, then there exists {strictly positive constants $ C_{u}, C_p$} independent of $h$ such that
\begin{equation}\label{priori1}
\begin{array}{ll}
||\u - \u_h||_{L^2(\Omega)}  \le C_u h, \\
||\nabla (p - p_h) ||_{3/2,\Omega} \le C_p h.
\end{array}
\end{equation}
}
\subsection{Iterative algorithm:} To compute the solution of the non-linear problem \eqref{V1h}, we introduce the following iterative algorithm:
for a given initial guess $\u_h^0 \in X_h$ and having $\u_h^i$ at each iteration $i$, we compute $(\u_h^{i+1}, p^{i+1}_h)$ solution of
\begin{equation}\label{V1hi}
\left\{
\begin{array}{ll}
\medskip
\forall \v_h \in X_h, \quad \ds \int_\Omega \alpha (\u_h^{i+1} - \u_h^i) \cdot \v_h \, d\x +  \frac{\mu}{\rho} \int_\Omega K^{-1} \u^{i+1}_h \cdot \v_h \, d\x + \frac{\beta}{\rho} \int_\Omega |\u^i_h| \u^{i+1}_h \cdot \v_h \, d\x \\
\medskip
\hspace{5cm} + \ds \int_\Omega \nabla p^{i+1}_h \cdot \v_h \, d\x = \ds \int_\Omega \f \cdot \v_h \, d\x,\\
\forall q_h\in M_h, \quad \ds \int_\Omega \nabla q_h \cdot \u^{i+1}_h \, d\x = -\int_\Omega b q_h \, d\x + \int_\Gamma g  q_h \, ds,
\end{array}
\right.
\end{equation}
%
%
%
%

{
We next  address  the convergence of Scheme  \eqref{V1hi}. The analysis of convergence  has two principal steps. The first is to derive an uniform bound to the numerical iterative solution
with respect to the iteration index. Then the obtained 
bound is used to study the convergence of a linear explicit scheme.   }
%
{\thm \label{boundu1} Problem $(\ref{V1hi})$ admits a unique solution $(\u_h^{i+1},p_h^{i+1})\in X_h \times M_h$. Furthermore, if the initial value $\u_h^0$ satisfies  the condition
\begin{equation}\label{cond1f}
{|| \u_h^0 ||_{L^2(\Omega)} \le \ds L_1(\f,\u_{h,l}, \alpha),}
\end{equation}
where
\begin{eqnarray}\label{L1}
L_1(\f,\u_{h,l}, \alpha) = \ds \left( \ell_0+\ell_1 \alpha \right)^{\frac{1}{2}},
\end{eqnarray}
with 
\begin{eqnarray}
\ell_0= 
{\left( \frac{2\rho}{\mu K_m} \right)^2}
\Big(  \ds(\frac{3\rho}{2\mu K_m}+ \frac{1}{2}) ||\f||^2_{L^2(\Omega)} + (\frac{1}{2} + \frac{3\mu K_M^2}{2\rho K_m} )  ||\u_{h,l} ||^2_{L^2(\Omega)}   +\frac{4\beta}{3\rho} || \u_{h,l} ||^3_{L^3(\Omega)}\Big),\\
\ell_1= { \left( \frac{2\rho}{\mu K_m} \right)^2}  ||\u_{h,l} ||^2_{L^2(\Omega)},
\end{eqnarray}
and if $\alpha \geq \alpha^\star$ with $\alpha^\star = 4(\gamma_1 + \sqrt{\gamma_1^2+\gamma_2})^2$ where %
\[
\begin{array}{rcl}
\medskip
\gamma_1 &=& \ds \frac{8\beta K_M}{3\rho K_m}   C_I^3 h^{-\frac{d}{2}} \sqrt{\ell_1}, \\ \medskip
\gamma_2 &=& \ds \frac{3\beta^2}{2\rho \mu K_m} {  C_I^4} h^{-\frac{2d}{3}}|| \u_{h,l} ||^2_{L^3(\Omega)}
+ \frac{8 \beta}{3 \mu K_m}  {  C_I^3} h^{-\frac{d}{2}} \| \f \|_{L^2(\Omega)^2} \\
&& \ds + \frac{8 \beta K_M}{3 \rho K_m}  {  C_I^3} h^{-\frac{d}{2}} \sqrt{\ell_0} +
\frac{8 \beta^2}{3 \rho^2}   {  C_I^6 }    h^{-d}  \max(\frac{2\rho \ell_0}{\mu K_m},\ell_1),
\end{array}
\]
\noindent then the solution of Problem $(\ref{V1hi})$ satisfies the estimates
\begin{equation}\label{equatt2}
{|| \u_h^{i+1} ||_{L^2(\Omega)} \le \ds L_1(\f,\u_{h,l},\alpha),}
\end{equation}
and
\begin{equation}\label{equatt2a}
{|| \u_h^{i+1} ||^3_{L^3(\Omega)} \le \ds \frac{3\mu K_m}{2\beta}  L_1^2(\f,\u_{h,l},\alpha).}
\end{equation}
}
\noindent \textbf {Proof.} {  To prove the existence and  uniqueness of the solution of Problem $(\ref{V1hi})$ which is a square linear system in finite dimension, it suffices to show the uniqueness. For a given $\u_h^i$, let $(\u_{h1}^{i+1},p_{h1}^{i+1})$ and $(\u_{h2}^{i+1},p_{h2}^{i+1})$ two different solutions of Problem $(\ref{V1hi})$ and let $\w_h= \u_{h1}^{i+1} - \u_{h2}^{i+1}$ and $\xi_h = p_{h1}^{i+1} - p_{h2}^{i+1}$, then $(\w_h,\xi_h)$ is the solution of the following problem:
\begin{equation*}
\left\{
\begin{array}{ll}
\medskip
\forall \v_h \in X_h, \quad \ds \int_\Omega \alpha \w_h \cdot \v_h \, d\x +  \frac{\mu}{\rho} \int_\Omega K^{-1} \w_h \cdot \v_h \, d\x + \frac{\beta}{\rho} \int_\Omega |\u^i_h| \w_h \cdot \v_h \, d\x + \ds \int_\Omega \nabla \xi_h \cdot \v_h \, d\x = 0,\\
\forall q_h\in M_h, \quad \ds \int_\Omega \nabla q_h \cdot \w_h \, d\x = 0.
\end{array}
\right.
\end{equation*}
By taking $(\v_h,q_h)=(\w_h,\xi_h)$ and by remarking that $\ds \frac{\beta}{\rho} \int_\Omega |\u^i_h| |\w_h|^2 \, d\x \ge 0$, we obtain by using the properties of $K^{-1}$ the following bound:
\[
(\alpha +  \ds \frac {K_m \mu}{\rho}) ||\w_h||^2_{L^2(\Omega)}  \le 0.
\]
Thus, we deduce that $\w_h = 0$. The inf-sup condition \eqref{infsuph1} deduces that $\xi_h=0$ and then, we get  the uniqueness of the solution of Problem $(\ref{V1hi})$.\\
}
%
\noindent To prove the bound \eqref{equatt2}, we need first to bound the error $\|   \u_h^{i+1} - \u_h^i \|_{L^2(\Omega)}$ with respect to the previous value $\u_h^{i}$. The second equation of Problem $(\ref{V1hi})$ allows us to deduce the relation
\begin{equation}\label{reluiuip1}
\forall q_h \in M_h, \quad \ds \int_\Omega \nabla q_h \cdot (\u^{i+1}_h - \u_h^i) \, d\x  = 0.
\end{equation}
Then, the  first equation of (\ref{V1hi}) with $\v_h = \u_h^{i+1} - \u_h^i$ gives
\[
\ds \alpha \| \u_{h}^{i+1} -\u_{h}^i\|^2_{L^2(\Omega)} + \frac{\mu}{\rho} \int_\Omega K^{-1} \u_{h}^{i+1} \cdot (\u_{h}^{i+1} -\u_{h}^i) \, d\x + \frac{\beta}{\rho} \int_\Omega |\u_{h}^i|  \u_{h}^{i+1} \cdot (\u_{h}^{i+1} -\u_{h}^i) \, d\x =  \ds \int_\Omega \f \cdot (\u_{h}^{i+1} - \u_h^i) \, d\x.
\]
By inserting $\pm \u_h^i$ in the second and the third terms of the last equation we get
\[
\begin{array}{ll}
\medskip
\ds \alpha \| \u_{h}^{i+1} -\u_{h}^i\|^2_{L^2(\Omega)} + \frac{\mu}{\rho} \int_\Omega K^{-1} |\u_{h}^{i+1} -\u_{h}^i |^2 \, d\x +
\frac{\beta}{\rho} \int_\Omega |\u_{h}^i| |\u_{h}^{i+1} -\u_{h}^i |^2 \, d\x \\
\hspace{2cm} = \ds \int_\Omega \f \cdot (\u_{h}^{i+1} - \u_h^i) \, d\x - \frac{\mu}{\rho} \int_\Omega K^{-1} \u_{h}^i \cdot (\u_{h}^{i+1} -\u_{h}^i) \, d\x
- \frac{\beta}{\rho} \int_\Omega |\u_{h}^i|  \u_{h}^i \cdot (\u_{h}^{i+1} -\u_{h}^i) \, d\x.
\end{array}
\]
By using the properties of $K$, the Cauchy-Schwartz inequality and Relation \eqref{eq:inversin},
we get, by remarking that the third term of the last equation is non-negative, the bound
\[
\begin{array}{ll}
\medskip
\ds \alpha \| \u_{h}^{i+1} -\u_{h}^i\|^2_{L^2(\Omega)} + \frac{\mu K_m}{\rho} \| \u_{h}^{i+1} -\u_{h}^i\|^2_{L^2(\Omega)} \le
\| \f \|_{L^2(\Omega)} \| \u_{h}^{i+1} -\u_{h}^i\|_{L^2(\Omega)} \\
\hspace{4cm} + \ds \frac{\mu K_M}{\rho}  \|\u_{h}^i \|_{L^2(\Omega)} \|\u_{h}^{i+1} -\u_{h}^i\|_{L^2(\Omega)} +
\frac{\beta}{\rho} {  C_I^3 }  h^{-\frac{d}{2}}   \|\u_{h}^i\|^2_{L^2(\Omega)}  \|\u_{h}^{i+1} -\u_{h}^i\|_{L^2(\Omega)}.
\end{array}
\]
%
{ We simplify by $\|\u_{h}^{i+1} -\u_{h}^i\|_{L^2(\Omega)}$ to obtain}
{
\[
\begin{array}{ll}
\medskip
\ds (\alpha+\frac{\mu K_m}{2\rho}) \| \u_{h}^{i+1} -\u_{h}^i\|_{L^2(\Omega)}   \le
\ds \| \f \|_{L^2(\Omega)} 
+  \frac{\mu K_M}{\rho}  \|\u_{h}^i \|_{L^2(\Omega)}  +
\frac{\beta}{\rho} {  C_I^3 }  h^{-\frac{d}{2}}   \|\u_{h}^i\|^2_{L^2(\Omega)},
\end{array}
\]}
and then { we get} the following bound
{
\begin{equation}\label{bound11}
\ds  \| \u_{h}^{i+1} -\u_{h}^i\|_{L^2(\Omega)} \le L_2(\f, \| \u_h^i \|_{L^2(\Omega)}),
\end{equation}}
where {$$ L_2(\f, \eta) = \ds \frac{1}{\alpha+\frac{\mu K_m}{2\rho} } \Big(\| \f \|_{L^2(\Omega)}
+  \frac{\mu K_M}{\rho}  \eta +
\frac{\beta}{\rho} {  C_I^3 }  h^{-\frac{d}{2}}\eta^2  \Big), \;\; \eta\in \mathbb R_+. $$}
Then, we are in position to show the relation \eqref{equatt2}. Property \eqref{equat11} allows us to deduce that the term
$\u^{i+1}_{h,0} = \u_h^{i+1} - \u_{h,l}$ is in $X_h$, and verifies
\begin{equation}\label{equat111}
\forall q_h \in M_h, \quad \ds \int_\Omega \nabla q_h \cdot \u^{i+1}_{h,0} \, d\x  = 0.
\end{equation}
We consider the first equation of (\ref{V1hi}) with $\v_h = \u^{i+1}_{h,0} = \u_h^{i+1} - \u_{h,l}$ and we obtain:
\[
\begin{array}{ll}
\medskip
\ds \alpha \int_\Omega (\u_{h}^{i+1} -\u_{h}^i)  \cdot  \u_{h}^{i+1} \, d\x +  \frac{\mu}{\rho} \int_\Omega K^{-1} |\u_{h}^{i+1}|^2 \, d\x + \frac{\beta}{\rho} \int_\Omega |\u_{h}^{i+1}|^3  \, d\x =  \ds \int_\Omega \f \cdot (\u_{h}^{i+1} - \u_{hl}) \, d\x \\
\medskip
\hspace{1cm} \ds + \alpha \int_\Omega (\u_{h}^{i+1} -\u_{h}^i)  \cdot  \u_{h,l} \, d\x +  \frac{\mu}{\rho} \int_\Omega K^{-1} \u_{h}^{i+1} \cdot \u_{h,l}  \, d\x + \frac{\beta}{\rho} \int_\Omega (|\u_h^{i+1}| - |\u^i_h|) |\u_{h}^{i+1}|^2  \, d\x\\
\hspace{2cm} \ds + \frac{\beta}{\rho} \int_\Omega (|\u^i_h| - |\u^{i+1}_h|)\u_{h}^{i+1} \cdot \u_{h,l} \, d\x + \frac{\beta}{\rho} \int_\Omega |\u^{i+1}_h| \,\u_{h}^{i+1} \cdot \u_{h,l} \, d\x.
\end{array}
\]
By using the properties of $K$, the Cauchy-Schwartz inequality
and the relation $a^2 b \le \frac{1}{3} \big( \frac{1}{\delta^3} b^3 + 2 \delta^{3/2} a^3 \big)$ (for any positive real numbers $a$ and $b$),
we get:
\[
\begin{array}{ll}
\medskip
\ds \frac{\alpha}{2} ||\u_{h}^{i+1}||^2_{L^2(\Omega)} - \frac{\alpha}{2} ||\u_{h}^i||^2_{L^2(\Omega)} + \frac{\alpha}{2} ||\u_{h}^{i+1} -\u_{h}^i||^2_{L^2(\Omega)} +  \frac{\mu}{\rho} K_m ||\u_{h}^{i+1}||^2_{L^2(\Omega)} + \frac{\beta}{\rho}  || \u_{h}^{i+1} ||^3_{L^3(\Omega)}\\
\medskip
\hspace{1cm} \le \ds ||\f||_{L^2(\Omega)} ||\u_{h}^{i+1} ||_{L^2(\Omega)} + ||\f||_{L^2(\Omega)} ||\u_{h,l}||_{L^2(\Omega)} + \alpha ||\u_{h}^{i+1} -\u_{h}^i||_{L^2(\Omega)} ||\u_{h,l} ||_{L^2(\Omega)}\\
\medskip
\hspace{1.5cm} + \ds \frac{\mu}{\rho} K_M || \u_{h}^{i+1} ||_{L^2(\Omega)} || \u_{h,l} ||_{L^2(\Omega)} + \frac{\beta}{\rho} || \u^i_h - \u^{i+1}_h ||_{L^3(\Omega)} || \u_{h}^{i+1} ||_{L^3(\Omega)} || \u_{h,l} ||_{L^3(\Omega)}  \\
\hspace{2cm}  \ds  + \frac{\beta}{\rho} || \u^i_h - \u^{i+1}_h ||_{L^3(\Omega)} || \u_{h}^{i+1} ||^2_{L^3(\Omega)} + \frac{\beta}{\rho} ||\u^{i+1}_h ||^2_{L^3(\Omega)}  || \u_{h,l} ||_{L^3(\Omega)}.
\end{array}
\]
We deduce by using the relation \eqref{eq:inversin}, that for any positive numbers $\varepsilon_i,i=1 \dots,4$ and $\delta_j,j=1,2$, we have the following bound:
\[
\begin{array}{ll}
\medskip
\ds \frac{\alpha}{2} ||\u_{h}^{i+1}||^2_{L^2(\Omega)} - \frac{\alpha}{2} ||\u_{h}^i||^2_{L^2(\Omega)} + \frac{\alpha}{2} ||\u_{h}^{i+1} -\u_{h}^i||^2_{L^2(\Omega)} +  \frac{\mu}{\rho} K_m ||\u_{h}^{i+1}||^2_{L^2(\Omega)} + \frac{\beta}{\rho} || \u_{h}^{i+1} ||^3_{L^3(\Omega)}\\
\medskip
\hspace{1cm} \le \ds  \frac{1}{2 \varepsilon_1} ||\f||^2_{L^2(\Omega)} + \frac{1}{2} \varepsilon_1 ||\u_{h}^{i+1} ||^2_{L^2(\Omega)} + \frac{1}{2} ||\f||^2_{L^2(\Omega)} + \frac{1}{2} ||\u_{h,l} ||^2_{L^2(\Omega)}\\
\medskip
\hspace{1.5cm}  \ds + \frac{\alpha}{2 \varepsilon_2}  ||\u_{h}^{i+1} -\u_{h}^i||^2_{L^2(\Omega)} + \frac{\alpha}{2} \varepsilon_2 ||\u_{h,l}||^2_{L^2(\Omega)} + \frac{\mu^2}{2\rho^2 \varepsilon_3} K_M^2 || \u_{h,l} ||^2_{L^2(\Omega)} + \frac{1}{2} \varepsilon_3 || \u_{h}^{i+1} ||^2_{L^2(\Omega)}\\
\medskip
\hspace{1.5cm}  \ds + \frac{\beta^2}{2\rho^2 \varepsilon_4} {  C_I^4} h^{-\frac{2d}{3}}|| \u_{h,l} ||^2_{L^3(\Omega)} || \u^i_h - \u^{i+1}_h ||^2_{L^2(\Omega)} + \frac{1}{2} \varepsilon_4   || \u_{h}^{i+1} ||^2_{L^2(\Omega)} \\
\medskip
\hspace{1.5cm}  \ds +  \frac{\beta}{3\rho} \big(  (\frac{1}{\delta_1})^{3} {  C_I^3 h^{-d/2}} ||\u^{i+1}_h -\u_{h}^i ||^3_{L^2(\Omega)}   +  2 \delta_1^{3/2} || \u_{h}^{i+1} ||^3_{L^3(\Omega)} \big)   \\

\hspace{1.5cm}  \ds + \frac{\beta}{3\rho} \big( (\frac{1}{\delta_2})^{3}  || \u_{h,l} ||^3_{L^3(\Omega)}   +  2 \delta_2^{3/2}||\u^{i+1}_h ||^3_{L^3(\Omega)} \big).
\end{array}
\]
We choose $\varepsilon_1=\varepsilon_3=\varepsilon_4=\ds\frac{\mu K_m}{3\rho}$, $\varepsilon_2=2$, $\delta_1=\delta_2=\ds \big(\frac{1}{2}\big)^{2/3}$ and we denote
$$
C_1 (\| \u_h^i \|_{L^2(\Omega)}) = \ds \frac{\alpha}{4} -  \frac{3\beta^2}{2\rho \mu K_m} {  C_I^4} h^{-\frac{2d}{3}}|| \u_{h,l} ||^2_{L^3(\Omega)}
-  \frac{4 \beta}{3 \rho}  { C_I^3 h^{-d/2} L_2(\f, \| \u_h^i \|_{L^2(\Omega)} }),
$$
which is not necessarily positive at this level.\\
By using the bound \eqref{bound11} , we get the following bound
\begin{equation}\label{relata1}
\begin{array}{ll}
\medskip
\ds \frac{\alpha}{2} ||\u_{h}^{i+1}||^2_{L^2(\Omega)} - \frac{\alpha}{2} ||\u_{h}^i||^2_{L^2(\Omega)} + C_1(\| \u_h^i \|_{L^2(\Omega)}) ||\u_{h}^{i+1} -\u_{h}^i||^2_{L^2(\Omega)} +  \frac{\mu K_m}{2 \rho}||\u_{h}^{i+1}||^2_{L^2(\Omega)} + \frac{\beta}{3\rho} || \u_{h}^{i+1} ||^3_{L^3(\Omega)}\\
\medskip
\hspace{1cm} \le \ds  (\frac{3\rho}{2\mu K_m}  + \frac{1}{2}) ||\f||^2_{L^2(\Omega)} + (\frac{1}{2} + \alpha + {\frac{3\mu K_M^2}{2\rho K_m} } )  ||\u_{h,l} ||^2_{L^2(\Omega)}   +\frac{4\beta}{3\rho} || \u_{h,l} ||^3_{L^3(\Omega)} \\
\hspace{1cm} \le \ds { \frac{\mu K_m}{2\rho } L_1^2(\f,\u_{h,l}, \alpha).}
\end{array}
\end{equation}
We now prove Estimate (\ref{equatt2}) by induction on $i$ under some conditions on $\alpha$. Starting with the relation \eqref{cond1f}, we suppose that we have
\begin{equation}\label{induction1}
|| \u_h^i ||_{L^2(\Omega)} \le \ds L_1(\f,\u_{h,l}, \alpha).
\end{equation}
We are in one of the following two situations :
\begin{itemize}
\item We have $\ds  ||\u_h^{i+1}||_{L^2(\Omega)} \le
||\u_h^i||_{L^2(\Omega)}$. We obviously deduce the bound
\[
||\u_h^{i+1}||_{L^2(\Omega)} \le \ds L_1(\f,\u_{h,l}, \alpha),
\]
from the induction hypothesis.
\item We have $\ds  ||\u_h^{i+1}||_{L^2(\Omega)} \ge
||\u_h^i||_{L^2(\Omega)}$. By using the induction condition \eqref{induction1} and the fact that the function $L_2$ is increasing with respect to $\eta$, we chose
\begin{eqnarray} \label{efg1}
\ds \frac{\alpha}{4} \ge  \ds  \frac{3\beta^2}{2\rho \mu K_m} {  C_I^4} h^{-\frac{2d}{3}}|| \u_{h,l} ||^2_{L^3(\Omega)}
+  \frac{4 \beta}{3 \rho}  {  C_I^3} h^{-d/2}{ L_2(\f, L_1(\f,\u_{h,l},\alpha))}
\end{eqnarray}
and we get
\begin{eqnarray}
\ds \frac{\alpha}{4} \ge   \ds  \frac{3\beta^2}{2\rho \mu K_m} {  C_I^4} h^{-\frac{2d}{3}}|| \u_{h,l} ||^2_{L^3(\Omega)}
+  \frac{4 \beta}{3 \rho}  {  C_I^3} h^{-d/2}   {L_2(\f,||\u_h^i||_{L^2(\Omega)} )},
\end{eqnarray}
which gives $C_1(||\u_h^i||_{L^2(\Omega)} ) \ge 0$ and then 
\[
||\u_h^{i+1}||_{L^2(\Omega)} \le \ds L_1(\f,\u_{h,l},\alpha).
\]
\end{itemize}
whence we deduce the relation \eqref{equatt2}. The bound \eqref{equatt2a} is a simple consequence of Equation \eqref{relata1} and relation \eqref{equatt2}. \\

We now focus on the inequality \eqref{efg1}. It is easy to show 
that, $\forall \eta \in \mathbb R_+$, 
%
\begin{equation}
 \ds L_2(\f, \eta) \leq
 \ds \frac{2 \rho}{\mu K_m} \| \f \|_{L^2(\Omega)^2} +  2 \eta + \frac{\beta}{\rho} C_I^3 h^{-\frac{d}{2}} \frac{\eta^2}{\alpha + \frac{\mu K_m}{2\rho}},
\end{equation}
and then to get by using the definition of $L_1$,
\[
L_2(\f, L_1(\f,\u_{h,l},\alpha)) \le \ds \frac{2 \rho}{\mu K_m} \| \f \|_{L^2(\Omega)^2} +  \frac{2K_M}{K_m} (\sqrt{\ell_0} +\sqrt{\ell_1} \sqrt{\alpha}) + \frac{2\beta}{\rho} C_I^3 h^{-\frac{d}{2}} 
\max(\frac{2\rho \ell_0}{\mu K_m},\ell_1).
\]
Relation \eqref{efg1} and the last inequality allow us to obtain
\begin{equation} \label{nouv324}
  \ds\frac{1}{4}\alpha-  \frac{3\beta^2}{2\rho \mu K_m} {  C_I^4} h^{-\frac{2d}{3}}|| \u_{h,l} ||^2_{L^3(\Omega)}
- \frac{4 \beta}{3 \rho}  {  C_I^3} h^{-\frac{d}{2}} { L_2(\f, L_1(\f,\u_{h,l},\alpha)) \geq \phi(\alpha)},
\end{equation}
{
with
\begin{equation} \label{defphi}
\phi(\alpha) =\frac{1}{4}\alpha-\gamma_1 \sqrt{\alpha}-\gamma_2,\end{equation}
where
\[
\begin{array}{rcl}
\medskip
\gamma_1 &=& \ds \frac{8\beta K_M}{3\rho K_m}   C_I^3 h^{-\frac{d}{2}} \sqrt{\ell_1}, \\ \medskip
\gamma_2 &=& \ds \frac{3\beta^2}{2\rho \mu K_m} {  C_I^4} h^{-\frac{2d}{3}}|| \u_{h,l} ||^2_{L^3(\Omega)}
+ \frac{8 \beta}{3 \mu K_m}  {  C_I^3} h^{-\frac{d}{2}} \| \f \|_{L^2(\Omega)^2}, \\
&& \ds + \frac{8 \beta K_M}{3 \rho K_m}  {  C_I^3} h^{-\frac{d}{2}} \sqrt{\ell_0} +
\frac{8 \beta^2}{3 \rho^2}   {  C_I^6 }    h^{-d}  \max(\frac{2\rho \ell_0}{\mu K_m},\ell_1).
\end{array}
\]

Therefore $\phi(\alpha)$ is a polynomial of second degree with respect to $\sqrt{\alpha}$ and admits the only positive root $\sqrt{\alpha^*} = 2(\gamma_1 + \sqrt{\gamma_1^2+\gamma_2})$. Thus we get that $\phi(\alpha) \ge0$ for all $\alpha \ge \alpha^*.$
}
%
%
$\hfill\Box$\\
\begin{Rem} \label{zero}
Notice  that Theorem \ref{boundu1} indicates that $\alpha^\star$ blows up as $h^{-d}$, and as {$\beta^2$} when respectively $h$ tends to zero, and $\beta$ approaches $+\infty$.
\end{Rem}

%
%

%
%

The next result provide the convergence of the solution $(\u_h^i, p_h^i)$ of Problem $(\ref{V1hi})$ in $L^2(\Omega)^d \times L^2(\Omega)$ to the unique solution
$(\u_h, p_h)$ of  the Problem $(\ref{V1h})$.
{\thm \label{converg1} Assume that there exists $\beta_0 >0$  such that, for every element $\kappa\in {\mathcal T}_h$, we have
\[ h_\kappa \ge \beta_0 h,\]
(which means that the family of triangulations is uniformly regular). Under the assumptions of Theorem $\ref{boundu1}$, and if $\alpha$ satisfies  the condition
\begin{equation}\label{condalpha}
\ds \alpha > \ds \max(\alpha^\star, \alpha^{\star \star} )
\end{equation}
 $\alpha^{\star\star}$ is the largest positive zero of the polynomial  function:
\begin{equation} \label{functiontildephi}
\tilde \phi(\alpha) = -\tilde \gamma_0 -\tilde \gamma_1 \alpha-\tilde \gamma_2 \alpha^2
+\frac{1}{8} \alpha^3,
\end{equation}
where $\tilde \gamma_i(h,\f, \u_{h,l}), i=2, 3, 4$ are strictly positive constants given by 
%
%
\begin{eqnarray*}
\tilde \gamma_0& =& \ell_0^2 \tilde{C}  h^{-2d},\\
\tilde \gamma_1&=& 2\ell_0\ell_1 \tilde{C}  h^{-2d},\\
\tilde \gamma_2 &=& \ell_1^2 \tilde{C}  h^{-2d},\\
\tilde{C} &=&  \frac{9 \beta^4 C_I^{12}}{32\mu K_m \rho^3}
\end{eqnarray*}
then the sequence of solutions $(\u_h^i, p_h^i)$ of Problem $(\ref{V1hi})$ converges in $L^2(\Omega)^d \times L^2(\Omega)$ to the solution $(\u_h, p_h)$ of Problem $(\ref{V1h})$.  }\\

\noindent \textbf {Proof.} We take the difference between the
equations (\ref{V1hi}) and (\ref{V1h}) with $\v_h = \u^{i+1}_h -\u_h$ and we obtain the equation
\[
\begin{array}{rr}
\medskip
\ds \frac{\alpha}{2} || \u_h^{i+1} - \u_h ||^2_{L^2(\Omega)}  - \frac{\alpha}{2} || \u_h^{i} - \u_h ||^2_{L^2(\Omega)} +
\frac{\alpha}{2} || \u_h^{i+1} - \u^i_h ||^2_{L^2(\Omega)} + \frac{\mu}{\rho} \int_\Omega K^{-1} ( \u_h^{i+1} - \u_h)^2 \, d\x \\
+ \ds \frac{\beta}{\rho} (|\u_h^{i}| \u_h^{i+1} - |\u_h| \u_h , \u_h^{i+1} - \u_h) = 0.
\end{array}
\]
The last term in the previous equation, denoted by $T$, can be decomposed
as
\[
T = \ds \frac{\beta}{\rho} ((|\u_h^{i}| - |\u^{i+1}_h|)  \u_h^{i+1}, \u_h^{i+1} - \u_h) + \frac{\beta}{\rho} (|\u_h^{i+1}| \u_h^{i+1} -
|\u_h| \u_h , \u_h^{i+1} - \u_h)= T_1+T_2.
\]
We denote by $T_1$ and $T_2$, respectively  the first and the
second terms in the right-hand side of the  last equation. Using \eqref{prop2}, we have $T_2 \ge 0$. Then we derive by
using  \eqref{equatt2a}, (\ref{eq:inversin}) and \eqref{KmM},
\[
\begin{array}{ll}
\medskip
\ds \frac{\alpha}{2} || \u_h^{i+1} - \u_h ||^2_{L^2(\Omega)}  - \frac{\alpha}{2} || \u_h^{i} - \u_h ||^2_{L^2(\Omega)} +
\frac{\alpha}{2}|| \u_h^{i+1} - \u^i_h ||^2_{L^2(\Omega)} + \frac{K_m \mu}{\rho}  \| \u_h^{i+1} - \u_h \|^2_{L^2(\Omega)} + T_2 \le  | T_1 |\\
\medskip
\hspace{1cm} \le \ds \frac{\beta}{\rho} \int_\Omega |\u_h^{i+1} - \u_h^{i}| \, |\u_h^{i+1}| \, |\u_h^{i+1} - \u_h| d{\textbf x}\\
\medskip
\hspace{1cm} \le \ds \frac{\beta}{\rho} ||\u_h^{i+1} - \u_h^{i} ||_{L^3(\Omega)} \; ||\u_h^{i+1}||_{L^3(\Omega)} \, ||\u_h^{i+1} - \u_h ||_{L^3(\Omega)}\\
\medskip
\hspace{1cm} \le \ds \frac{\beta}{\rho}   C_I^2  h^{-\frac{d}{3}}||\u_h^{i+1} - \u_h^{i} ||_{L^2(\Omega)} \; ||\u_h^{i+1}||_{L^3(\Omega)} \, ||\u_h^{i+1} - \u_h ||_{L^2(\Omega)}\\
\medskip
\hspace{1cm} \le \ds \frac{\beta}{\rho} \big( \frac{3\mu K_m}{2\beta}  \big)^{1/3}  C_I^2 h^{-\frac{d}{3}} L^{2/3}_1(\f,\u_{h,l},\alpha)  ||\u_h^{i+1} - \u_h^{i} ||_{L^2(\Omega)} \, ||\u_h^{i+1} - \u_h ||_{L^2(\Omega)}.
\end{array}
\]
We denote by $C= \ds \frac{\beta}{\rho}  \big( \frac{3\mu K_m}{2\beta}  \big)^{1/3}    C_I^2 
L^{2/3}_1(\f,\u_{h,l},\alpha)$ and we use the inequality $ab \le \ds
\frac{1}{2\varepsilon} a^2 + \frac{\varepsilon}{2} b^2$ (with
$\varepsilon = \ds \frac{K_m \mu}{\rho}$) to obtain the
following bound
\[
\begin{array}{ll}
\medskip
\ds \frac{\alpha}{2} || \u_h^{i+1} - \u_h ||^2_{L^2(\Omega)}  - \frac{\alpha}{2} || \u_h^{i} - \u_h ||^2_{L^2(\Omega)} +
\frac{\alpha}{2} || \u_h^{i+1} - \u^i_h ||^2_{L^2(\Omega)} + \frac{K_m \mu}{2\rho } {  || \u_h^{i+1} - \u_h ||^2_{L^2(\Omega)}} \\
\hspace{2cm} \le \ds \frac{\rho C^2}{2K_m \mu}  h^{-\frac{2d}{3}} || \u_h^{i+1} - \u_h^i ||^2_{L^2(\Omega)}.
\end{array}
\]
We choose
\begin{equation}\label{equat1}
\ds \frac{\alpha}{2} >  \ds \frac{\rho C^2}{2K_m \mu}  h^{-\frac{2d}{3}},
\end{equation}
denote by $C_1 =
\ds \frac{1}{2}(\alpha - \ds \frac{\rho C^2}{K_m \mu}   h^{-\frac{2d}{3}})$  and obtain
\begin{equation}\label{relat22}
\ds \frac{\alpha}{2} || \u_h^{i+1} - \u_h ||^2_{L^2(\Omega)}  - \frac{\alpha}{2} || \u_h^{i} - \u_h ||^2_{L^2(\Omega)} + C_1 ||
\u_h^{i+1} - \u^i_h ||^2_{L^2(\Omega)} + \frac{\mu K_m}{2\rho} {  || \u_h^{i+1} - \u_h ||^2_{L^2(\Omega)} } \le 0.
\end{equation}
We clearly have that \eqref{equat1} is satisfied if
\begin{equation}
    \tilde{\phi}(\alpha)> 0.
\end{equation}
Simple calculation shows that $\tilde  \phi^\prime(\alpha)$ is a polynomial function of degree two that is positive for large $\alpha$  and posses a unique positive zero. Since $\tilde{\phi}(0) <0$, we deduce that $\tilde{\phi}(\alpha)$ vanishes at least once on  $\mathbb R_+$. By denoting $\alpha^{\star\star}$ the largest root of $\tilde{\phi}(\alpha)$, we deduce that $\tilde{\phi}(\alpha) \ge 0$ for $\alpha \ge \alpha^{\star\star}$.
\\

We deduce from \eqref{relat22} that, for all $i\ge 1$, we have (if $|| \u_h^{i} - \u_h ||_{L^2(\Omega)} \neq 0$)

\[
|| \u_h^{i+1} - \u_h ||_{L^2(\Omega)}  < || \u_h^{i} - \u_h
||_{L^2(\Omega)},
\]
and we deduce the convergence of the sequence $(\u_h^{i+1} - \u_h)$ in $L^2(\Omega)^d$ and then the convergence of the sequence $\u_h^{i}$ in $L^2(\Omega)^d$. By taking the limit of (\ref{relat22}) and remarking that $T_2 \ge 0$, we get
\[
\underset{{\small i \rightarrow +\infty} }{\lim} \Big( {  || \u_h^{i+1} - \u_h ||^2_{L^2(\Omega)}}  \Big) \le 0.
\]
We deduce then that ${  || \u_h^{i+1} - \u_h ||_{L^2(\Omega)} }$ converges to $0$ and $\u_h^{i+1}$ converges to $\u_h$ in $L^2(\Omega)^d$.\\
For the convergence of the pressure, we take the difference between the equations (\ref{V1hi}) and (\ref{V1h}) and we obtain for all $\v_h \in X_h$ the equation
\[
\begin{array}{rr}
\medskip
\ds \int_\Omega \nabla (p_h^{i+1} - p_h) \v_h \, d\x = \ds -\alpha \int_\Omega (\u_h^{i+1} - \u^i_h)\v_h \, d\x  + \frac{\mu}{\rho} \int_\Omega K^{-1} ( \u_h - \u_h^{i+1}) \v_h \, d\x \\
\hspace{2cm} \ds  + \frac{\beta}{\rho} ( (|\u_h| - |\u_h^i |) \u_h, \v_h) + \frac{\beta}{\rho} ( |\u^i_h| (\u_h - \u_h^{i+1}) , \v_h).
\end{array}
\]
We get by using the inverse inequality \eqref{eq:inversin} the following:
\[
\begin{array}{rcl}
\medskip
\ds \frac{\ds \Big| \int_\Omega \nabla (p_h^{i+1} - p_h) \v_h \, d\x \Big| }{||\v_h ||_{L^3(\Omega)}} &\le&  \ds (\alpha || \u_h^i - \u_h^{i+1}||_{L^2(\Omega)} + \frac{\mu K_m}{\rho}  || \u_h - \u_h^{i+1}||_{L^2(\Omega)}) \frac{||\v_h||_{L^2(\Omega)}}{||\v_h||_{L^3(\Omega)}} \\
&& \ds  { + \frac{\beta}{\rho}  C_I h^{-\frac{d}{6}}  ||\u_h - \u_h^i||_{L^2(\Omega)} (||\u_h ||_{L^3(\Omega)}  + ||\u_h^i||_{L^3(\Omega)} ).}
\end{array}
\]
Owning the inf-sup condition \eqref{infsuph1}, we deduce the relation
\[
\begin{array}{ll}
\medskip
\ds || \nabla (p_h^{i+1} - p_h) ||_{L^{3/2}(\Omega)} \le  \ds \frac{1}{\beta_u}\Big( \alpha {  |\Omega|^{1/6}} || \u_h^i - \u_h^{i+1}||_{L^2(\Omega)} + {  |\Omega|^{1/6}} \frac{\mu K_m}{\rho}  || \u_h - \u_h^{i+1}||_{L^2(\Omega)}\\
 \hspace{2cm} \ds {  + \frac{\beta}{\rho}  C_I h^{-\frac{d}{6}}  ||\u_h - \u_h^i||_{L^2(\Omega)} (||\u_h ||_{L^3(\Omega)}  + ||\u_h^i||_{L^3(\Omega)} )\Big).}
\end{array}
\]
Thus for a given mesh (given $h$), the strong convergence of $\u_h^i$ to $\u_h$ in $L^2(\Omega)^d$ implies the strong convergence of $\nabla p_h^i$ to $\nabla p_h$ in $L^{\frac{3}{2}}(\Omega)$. Furthermore, the fact that $p_h^i$ and $p_h$ are in the discrete space of $\P_1$ finite elements  $M_h\subset L^2_0(\Omega)$ allows us to deduce the strong convergence of $p_h^i$ to $p_h$ in $L^2(\Omega)$ .$\hfill\Box$

\begin{Rem} \label{zero2}
One can show that 
\begin{align*}
\alpha^{\star\star} \geq \max\left(8\ell_1^2 C_h, 4(\ell_0\ell_1 C_h)^{\frac{1}{2}},
2(\ell_0^2C_h)^{\frac{1}{3}}
\right),   
\end{align*}
where $C_h = \left(\frac{\beta^2 C_I^{4}  }{2K_m \mu\rho}\right)^3 h^{-2d}$. 
Consequently  $\alpha^{\star\star}$ tends to infinity  when $h$ becomes small  or $\beta$ approaches  $+\infty$. 
\end{Rem}

{\rmq The condition $(\ref{cond1f})$  supposes that the initial value of the algorithm is small related to the data $f$. We can for example consider $u_h^0=0$.}

{\rmq \label{remtheoralpha}
Theorems \ref{converg1} and \ref{boundu1} require the conditions $\alpha > \max(\alpha^\star, \alpha^{\star\star})$ and \eqref{condalpha} to get the convergence of the numerical scheme (\ref{V1hi}). This conditions cannot be computed easily in practice, especially for the numerical investigations. In fact, This result of convergence states that for a given mesh (for a given $h$), the iterative solution $(u_h^i, p_h^i)$ converges to $(u_h,p_h)$ when $i \rightarrow +\infty$. 
}
\subsection{{\it A posteriori} error estimates}
\noindent As usual, for {\it a posteriori } error estimates, we introduce the following notations. We denote by
\begin{itemize}
\item $\Gamma_h^i$ the set of edges (when $d=2$) or faces (when
$d=3$) of $\kappa$ that are not contained in $\partial \Omega$.
\item $\Gamma_h^b$  the set of edges (when $d=2$) or faces (when
$d=3$) of $\kappa$ which are contained in $\partial \Omega$.
\end{itemize}
For every element $\kappa$ in  $\mathcal{T}_h$, we denote by $w_\kappa$ the union of elements $K$ of $\mathcal{T}_h$ such that $\kappa \cap K \ne \phi$. Furthermore, for every edge (when $d=2$) or face (when $d=3$) $e$ of the mesh
 $\mathcal{T}_h$, we denote by
\begin{itemize}
\item $\omega_e$ the union of elements of $ \mathcal{T}_h$  adjacent to
$e$.
\item $[\cdot]_e$ the jump through $e\in \Gamma_h^i$. 
\end{itemize}
From now on, to simplify, we set $d=3$.  The extension to two dimensions is straightforward  and simpler. We suppose also that $b \in L^3(\Omega)$ and $g  \in L^3(\Gamma)$.\\

In this and the next  sections, the {\it a posteriori} error estimates are established when the solution is slightly smoother.\\

Let $R_h$ be a Cl\'ement-type interpolation operator \cite{s3}. 
We have the following error estimate: for all $\kappa$ in  $\mathcal{T}_h$, for all $e$ in $\partial \kappa$ and for all $ q \in W^{1,3/2}(\Omega)$,
\begin{equation}\label{estom}
\| q - R_h q \|_{L^{3/2}(\kappa)} \le C_\kappa h_\kappa |q |_{W^{1,3/2}(w_\kappa)}
\end{equation}
and
\begin{equation}\label{estgam}
\| q - R_h q \|_{L^{3/2}(e)} \le C_e h_e^{1/3} |q |_{W^{1,3/2}(w_e)},
\end{equation}
where $C_e$ and $C_\kappa$ are positive constants independent of $h$.
{ \rmq Relations \eqref{estom} and \eqref{estgam} are a direct consequence of the following two properties:
\begin{itemize}
\item for all integers $\ell$, $0\le \ell \le 2$, and for all $p$, $0\le p \le +\infty$, there exists a constant $C$, independent of $h_\kappa$, such that for all $\kappa \in \mathcal{T}_h$ and all function  $q\in W^{\ell,p}(w_\kappa)$, the following inequalities hold (see for instance \cite{belhachimi}, Theorem 1):
\[
||q - R_h q ||_{L^p(\kappa)} \le C h_\kappa^\ell |q|_{W^{\ell,p}(w_\kappa)}
\]
and when $\ell \ge 1$
\[
|q - R_h q |_{W^{1,p}(\kappa)} \le C h_\kappa^{\ell-1} |q|_{W^{\ell,p}(w_\kappa)}.
\]
\item Let $s\in ]0,1[$ and $p >\ds \frac{1}{s}$ with $p\in [1, +\infty[$ or $s=1$ with $p\in [1, +\infty]$. Then there exists $c$, uniform with  respect to the mesh such that the following trace inequality (see \cite{ernguermont}, Lemma 7.2) holds for all $q\in W^{s,p}(\kappa)$ and all $\kappa \in  \mathcal{T}_h$:
\[
||q||_{L^{p}(e)} \le c \big( h_\kappa^{-\frac{1}{p}} ||q||_{L^{p}(\kappa)} + h_\kappa^{s-\frac{1}{p}}  |q|_{W^{s,p}(\kappa)} \big).
\]
\end{itemize}
}
\subsubsection{Upper error bound}
\noindent In order to establish upper bounds, we introduce, on every edge  ($d=2$) or face ($d=3$)  $e$ of the mesh, the function
\begin{equation}
\label{eq:phih}
 \phi_{h,1}^e=
 \left\{
\begin{array}{ccl}
\medskip
\ds \frac{1}{2}\,[\textbf{u}_h^{i+1}\cdot\textbf{n}]_{e} \quad \mbox{if } e \in \Gamma_h^i,\\
 \textbf{u}_h^{i+1}\cdot\textbf{n} - g_h \quad \mbox{if } e  \in \Gamma_h^b,
\end{array}
\right.
\end{equation}
where $g_h$ is an approximation of $g $ which is constant on each $e$.\\

\noindent  A standard calculation shows that  the solutions of the problems \eqref{V1} and \eqref{V1hi}
verify for all $(\v,q) \in X  \times M $  and $(\v_h,q_h) \in X_h \times M_h$:
\begin{equation}\label{1.71}
\begin{array}{ll}
\medskip
 \ds  \frac{\mu}{\rho} \int_\Omega K^{-1} (\u - \u^{i+1}_h) \cdot \v \, d\x +
 \frac{\beta}{\rho} \int_\Omega (|\u| \u - |\u^i_h| \u^{i+1}_h) \cdot \v \, d\x + \ds\int_{\Omega} \nabla  ( p-p_h^{i+1} ) \cdot \textbf{v}\,d\x\\
 \medskip
 \hspace{1cm}  = \ds \sum_{\kappa \in  \mathcal{T}_h } \big[ \int_{\kappa} (-\nabla p_h^{i+1}- {{\alpha (\u_h^{i+1} - \u_h^i)}} - {{
\frac{\mu}{\rho}K^{-1} \u^{i+1}_h }}- \frac{\beta}{\rho} |\u^i_h| \u^{i+1}_h + \f_h) \cdot(\v- \v_h)\,d\x\, \\
\medskip
\hspace{5cm} + \ds \int_{\kappa}
(\f-\f_h)\cdot \v\,d\x\,  + \alpha \ds \int _{\kappa} (\u_h^{i+1} - \u_h^i) \cdot \v \, d\x \big],
\end{array}
\end{equation}

and by using the second equations of the systems \eqref{V1} and \eqref{V1hi}, the fact that $\dd \, \u_h^{i+1}=0$ (as in each element $\kappa$ we have: $\u_h^{i+1} \in \P_0^d(\kappa)$),  and by applying the integration by parts to the term $\ds \int_\Omega \nabla (q-q_h) \cdot \u_h^{i+1} d\x$ and using the definition of $\phi_{h,1}^e$ we get,
\begin{equation}\label{1.72}
\begin{array}{rcl}
\medskip
\ds \int_{\Omega} \nabla q \cdot
(\textbf{u}-\textbf{u}_h^{i+1})\,d\x\,&=&\ds
\int_{\Omega} \nabla q \cdot
\textbf{u}\,d\x - \int_{\Omega} \nabla q_h \cdot \textbf{u}_h^{i+1}\,d\x - \int_{\Omega} \nabla (q - q_h) \cdot \textbf{u}_h^{i+1}\,d\x\\
\medskip
&=& \ds -\sum_{\kappa \in  \mathcal{T}_h }
\Big[\int_{\kappa} (q-q_h) (b - b_h) \,d\x\, + \int_{\kappa} b_h (q-q_h)  \,d\x\\
&& \qquad \qquad \ds +\ds \sum_{e \in \partial \kappa} \int _{e} \phi_{h,1}^e(q-q_h)\,d\s\,\Big]
-  \sum_{e \in \Gamma_h^b} \int _{e} (g_h - g ) (q-q_h) \,d\s,
\end{array}
\end{equation}
where $\f_h$ (resp. $b_h$) is an approximation of $\f$ (resp. $b$) which is constant on each element $\kappa$ of  $\mathcal{T}_h$.\\
From the error equations \eqref{1.71} and \eqref{1.72}, we deduce the following error indicators  for
each  $\kappa \in  \mathcal{T}_h$,
\begin{equation}\label{indicatorss}
\begin{array}{rcl}
\medskip
\ds \eta_{\kappa,i}^{(L)}&=&||\u_h^{i+1}-\u_h^{i}||_{L^2(\kappa)},\\
\medskip
\ds \eta_{\kappa,i}^{(D_1)}&=&\ds \|-\nabla p_h^{i+1} - {{\alpha (\u_h^{i+1} - \u_h^i) - \frac{\mu}{\rho} K^{-1} \u^{i+1}_h}} - 
\frac{\beta}{\rho} |\u^i_h| \u^{i+1}_h + \f_h \, \|_{L^2(\kappa)},
\\
\medskip
\ds \eta_{\kappa,i}^{(D_2)} &=& h_\kappa \|b_h \|_{L^3(\kappa)}  +\ds  \sum_{e \in \partial
\kappa} h_e^{\frac{1}{3}} \|\phi_{h,1}^e\|_{L^3(e)}.
\end{array}
\end{equation} 

The term $h_\kappa \norm{b_h}_{L^3(\kappa)}$ which appears in $\eta_{\kappa,i}^{(D_2)}$ is an indicator since it represents the quantity $h_\kappa(\norm{b_h - \div(\u_h^{i+1})}_{L^3(\kappa)})$ as $\div(\u_h^{i+1})=0$ in each element $\kappa$.\\

In order to establish an {\it a posteriori} error estimate, we need to bound the numerical solution $\u_h^{i+1}$ in $L^6(\Omega)$ which is the subject of the next lemma.
\begin{lem}\label{bornl3}
Let $d=3$ and let the mesh satisfy \eqref{eq:reg}. Under the assumptions of Theorems \ref{thmpriori1}, \ref{boundu1} and \ref{converg1}, and if the exact velocity $\u \in W^{1.6}(\Omega)^d$, there exists an integer $i_0$ depending on $h$ such that for all $i\ge i_0$,  the solution $(\u_h^{i+1},p_h^{i+1})$ of $\eqref{V1hi}$ verifies the following bound:
\begin{equation}\label{b3}
\|\u^{i+1}_h \|_{L^6(\Omega)} \le \widehat{C}(\u,p),
\end{equation}
where $\widehat{C}$ is a constant depending on the exact solution $(\u,p)$ of \eqref{V1}.
\end{lem}
\begin{proof} We consider the case $d=3$.  Let $(\u,p)$ be the solution of \eqref{V1}, $(\u_h,p_h)$ the solution of $\eqref{V1h}$ and $(\u_h^{i+1},p_h^{i+1})$ the solution of $\eqref{V1hi}$. \\
By using \eqref{eq:inversin} (for $p=6$), the properties of the operator $R_h$ and the {\it a priori} error estimate \eqref{priori1},
the term $\|\u^{i+1}_h \|_{L^6(\Omega)}$ can be bounded as following:
\begin{equation}\label{equat55}
\begin{array}{rcl}
\medskip
\|\u^{i+1}_h \|_{L^6(\Omega)} &\le& \|\u^{i+1}_h - \u_h\|_{L^6(\Omega)} + \|\u_h - R_h(\u)\|_{L^6(\Omega)}  + \|R_h(\u) - \u  \|_{L^6(\Omega)} + \|\u \|_{L^6(\Omega)} \\
\medskip
&\le & C   \big( h^{-1} (\|\u^{i+1}_h - \u_h\|_{L^2(\Omega)} +  \|\u_h - R_h(\u)\|_{L^2(\Omega)}) + h \| \u \|_{W^{1,6}(\Omega)^6} \big)  + \|\u \|_{L^6(\Omega)}\\
&\le & C   \big( h^{-1} (\|\u^{i+1}_h - \u_h\|_{L^2(\Omega)} +  \|\u_h - \u \|_{L^2(\Omega)} + \| \u -  R_h(\u)\|_{L^2(\Omega)}  \big) + h \| \u \|_{W^{1,6}(\Omega)} \big)  \\ \medskip \medskip
&& \ds  \hspace{.0cm}  + \|\u \|_{L^6(\Omega)}\\
&\le& C ( h^{-1} \|\u^{i+1}_h - \u_h\|_{L^2(\Omega)} + C_u(\u,p) + |\u|_{W^{1,2}(\Omega)} + h  \| \u \|_{W^{1,6}(\Omega)}  + \| \u \|_{L^6(\Omega)} ).
\end{array}
\end{equation}
As  $\u^{i+1}_h$ converges to $\u_h$ in $L^2(\Omega)$, then there exists an integer $i_0$ depending on $h$ such that for all $i\ge i_0$ we have
\begin{equation}\label{equat66}
\| \u^{i+1}_h - \u_h \|_{L^2(\Omega)} \le h.
\end{equation}
Then, Equation \eqref{equat55} gives by using \eqref{equat66} the following bound for all $i\ge i_0$:
\begin{equation}\label{equat77}
\|\u^{i+1}_h \|_{L^6(\Omega)} \le \widehat{C}(\u,p).
\end{equation}
\end{proof}
Our main goal is to get an upper bound of the error between the exact solution $(\u,p)$ of \eqref{V1} and the numerical solution $(\u_h^{i+1},p_h^{i+1})$ of $\eqref{V1hi}$. To get this desired result, we start by the following lemma which can be proved by using the inf-sup condition \eqref{inf-supS1}.
\begin{lem}
There exists a velocity $\v_r$ in $X$ that solves the following equation : $\forall q \in M$
\begin{equation} \begin{aligned} \label{3.30}
\int_{\Omega} \nabla q \cdot \v_r\,d\x &= \sum_{\kappa \in \mathcal{T}_h} \big[ \int_{\kappa}(q-R_h(q))(-b+b_h)\,d\x - \int_{\kappa} b_h(q-R_h(q))\,d\x \\
&- \sum_{e \in \partial \kappa} \int_{e} \phi_{h,1}^e (q-R_h(q))ds \big] - \sum_{e \in \Gamma_h^b}\int_{e}(g_h-g)(q-R_h(q))ds
\end{aligned}
\end{equation}
satisfying the following bound
\begin{equation}
\label{eq:liftq2bdd1}
\begin{array}{rcl}
\medskip
\|\v_r\|_{L^3(\Omega)} &\le& \ds \hat{C}_2\big(  \sum_{\kappa \in  \mathcal{T}_h } \big[ (\eta_{\kappa,i}^{(D_2)}) + h_\kappa || b - b_h||_{L^3(\kappa)} \big] + \sum_{e \in \Gamma_h^b} h_e^{1/3} \| g_h - g  \|_{L^3(e)}  \big).
\end{array}
\end{equation}
\end{lem}
\begin{proof}
Equation \eqref{1.72} with $q_h= R_h(q)$ and the inf-sup condition \eqref{inf-supS1} imply that there exists a $\v_r \in X$ such that \eqref{3.30} is verified and satisfying
\begin{equation}\label{vrr}
\begin{array}{rcl}
\medskip
\|\textbf{v}_r\|_{L^3(\Omega)} &\le& \underset{q \in M }{\sup} \frac{1}{|\nabla q|_{L^{3/2}(\Omega)}} \Big|
\ds \sum_{\kappa \in  \mathcal{T}_h }
\big[\int_{\kappa} (q- R_h(q))( - b + b_h)\,d\x\, -
\int_{\kappa} b_h (q- R_h(q)) \,d\x,\\
\medskip
&& \ds -
\ds \sum_{e \in \partial \kappa} \int _{e} \phi_{h,1}^e(q-R_h(q))\,d\s\,\big]
-  \sum_{e \in \Gamma_h^b} \int _{e} (g_h - g ) (q-R_h(q)) \,d\s\Big|\\
\medskip
&& \hspace{-2cm}  \le  \underset{q \in M }{\sup} \frac{1}{|\nabla q|_{L^{3/2}(\Omega)}} \Big|
\ds \sum_{\kappa \in  \mathcal{T}_h } \|q- R_h(q) \|_{L^{3/2}(\kappa)} \, \| b - b_h \|_{L^3(\kappa)} +  \|q- R_h(q) \|_{L^{3/2}(\kappa)} \, \| b_h \|_{L^3(\kappa)} \\
&& \hspace{-1cm} \ds + \ds \sum_{e \in \partial \kappa} \| \phi_{h,1}^e \|_{L^3(e)} \, \| q-R_h(q) \|_{L^{3/2}(e)} \big]
+  \sum_{e \in \Gamma_h^b} \|g_h - g  \|_{L^3(e)} \| q-R_h(q) \|_{L^{3/2}(e)} \Big|.
\end{array}
\end{equation}
Thus, from the properties of the operator $R_h$, the regularity of $\mathcal{T}_h$ and the following Holder's inequality ($p=3/2$, $q=3$)
\begin{equation}\label{holderr}
\ds \sum_{k=1}^n a_k b_k \le \ds \big(\sum_{k=1}^n a_k^p \big)^{1/p} \big(\sum_{k=1}^n b_k^q \big)^{1/q},
\end{equation}
we infer after cubing the last relation: 
\begin{equation*}
\begin{array}{rcl}
\medskip
\|\textbf{v}_r\|^3_{L^3(\Omega)} &\le& \ds C_2\big(  \sum_{\kappa \in  \mathcal{T}_h } \big[ (\eta_{\kappa,i}^{(D_2)})^3 + h_\kappa^3 || b - b_h||^3_{L^3(\kappa)} \big] + \sum_{e \in \Gamma_h^b} h_e \| g_h - g  \|^3_{L^3(e)}  \big).
\end{array}
\end{equation*}
Finally, we obtain \eqref{eq:liftq2bdd1} by taking the cubic root of the previous inequality.
\end{proof}
%
%

%
%
\begin{thm}
\label{aposterioril2}
Under the assumptions of Lemma \ref{bornl3}, there exists an integer $i_0$ depending on $h$ such that for all $i\ge i_0$,  the solutions $(\u,p)$ of \eqref{V1} and $(\u_h^{i+1},p_h^{i+1})$ of $\eqref{V1hi}$ verify the following error inequalities:
\begin{equation}
\label{eq:uperr2}
\begin{split}
\|\v_r\|_{L^3(\Omega)} + \|\u-\u_h^{i+1}\|_{L^2(\Omega)} + \left\|\u-\u_h^{i+1} \right\|_{L^3(\Omega)^d}^{3/2} \leq C\ds  \Big[ \sum_{\kappa \in  \mathcal{T}_h} \Big( \eta_{\kappa,i}^{(D_1)} + \eta_{\kappa,i}^{(D_2)} + \eta_{\kappa,i}^{(L)} \Big) \\
 \hspace{2cm} + \sum_{\kappa \in  \mathcal{T}_h}  \Big( \|\textbf{f}-\textbf{f}_h\|_{L^2(\kappa)} + h_k \|b - b_h \|_{L^3(\kappa)}  + \sum_{e \in \Gamma_h^b} h_e^{\frac{1}{3}} ||g_h - g ||_{L^3(e)}  \Big)\Big],
\end{split}
\end{equation}
where $ \mathbf{z_0}=\u-\u_h^{i+1}-\v_r$ and $C$ is a constant depending on $\u$, $\f_h$ (resp. $b_h$) is an approximation of $\f$ (resp. $b$) which is constant on each element $\kappa$ of  $\mathcal{T}_h$, and $g_h$ is an approximation of $g$ which is constant on each face $e$ of $\mathcal{T}_h \cap \Gamma$.

\end{thm}
\begin{proof}
The velocity error equation \eqref{1.71} can be written as
\begin{equation}\label{eq:velocierr2}
\begin{array}{ll}
\medskip
 \ds  \frac{\mu}{\rho} \int_\Omega K^{-1} (\u - \u^{i+1}_h) \cdot \v \, d\x +
 \frac{\beta}{\rho} \int_\Omega (|\u| \u - |\u^{i+1}_h| \u^{i+1}_h) \cdot \v \, d\x
  + \ds\int_{\Omega} \nabla  ( p-p_h^{i+1} ) \cdot \textbf{v}\,d\x \\
 \medskip
 \ds = -  \frac{\beta}{\rho}  \int_\Omega ((|\u^{i+1}_h|-|\u^i_h|) \u^{i+1}_h) \cdot \v \, d\x + \ds  \alpha \int_\Omega (\u_h^{i+1} - \u_h^{i}) \cdot \v \, d\x \\
 \medskip
 + \ds \sum_{\kappa \in \mathcal{T}_h } \big[ \int_{\kappa} (-\nabla p_h^{i+1}-{{
 \frac{\mu}{\rho}K^{-1} \u_h^{i+1}}}-
 \alpha (\u_h^{i+1}-\u_h^i)-\frac{\beta}{\rho}|\u_h^i|\u_{h}^{i+1}+\f_h)\cdot (\v-\v_h)\,d\x + \int_{\kappa}(\f-
 \f_h) \cdot \v \,d\x
 \big].
\end{array}
\end{equation}
\noindent Now, to simplify we set $\textbf{z}_0 = \textbf{u}-\textbf{u}_h^{i+1}- \textbf{v}_r$ and we test \eqref{eq:velocierr2} with $\textbf{v}=\textbf{z}_0$ and $\textbf{v}_h ={\bf 0}$. \\
\noindent By construction, \eqref{3.30} and (\ref{1.72}) imply that we have with $q_h=R_hq$
\begin{equation}\label{perpond}
\mbox{for all } q\in M, \qquad \qquad  
\int_\Omega \nabla q  \cdot \textbf{z}_0\,d\x =0.
\end{equation}
Hence \eqref{eq:velocierr2} reduces to
\begin{equation} \begin{aligned} \label{upbdvit5}
 \frac{\mu}{\rho} & \int_{\Omega} K^{-1}\mathbf{z_0} \cdot \mathbf{z_0} \,d\x +
 \frac{\mu}{\rho} \int_{\Omega}K^{-1}\v_r \cdot \mathbf{z_0} \,d\x \\
 &+ \frac{\beta}{\rho} \int_{\Omega} ( |\u|\u - |\u_h^{i+1}|\u_h^{i+1}| ) ( \u - \u_h^{i+1})\,d\x
 -\frac{\beta}{\rho}\int_{\Omega} ( |\u|\u - |\u_h^{i+1}|\u_h^{i+1})\cdot \v_r \,d\x \\
 &= \alpha\int_{\Omega} (\u_h^{i+1}-\u_h^i)\cdot \mathbf{z_0} \,d\x -
 \frac{\beta}{\rho} \int_{\Omega} \big( ( |\u_h^{i+1}|-|\u_h^i| ) \u_h^{i+1} \big) \cdot \mathbf{z_0} \,d\x \\
 &+ \sum_{\kappa \in \mathcal{T}_h} \Big[
\int_{\kappa} \big( - \nabla p_h^{i+1} -\alpha (\u_h^{i+1}-\u_h^i) - \frac{\mu}{\rho} K^{-1}\u_h^{i+1}- \frac{\beta}{\rho} |\u_h^i|\u_h^{i+1}+\f_h \big) \cdot \mathbf{z_0}\,d\x 
+\int_{\kappa} \big(\f-\f_h\big)\cdot \mathbf{z_0} \,d\x\Big].
\end{aligned}     
 \end{equation}
%
%
We decompose the fourth term of the left hand side as following
\[
\ds \frac{\beta}{\rho} \int_\Omega (|\u| \u - |\u^{i+1}_h| \u^{i+1}_h) \cdot \v_r \, d\x = \ds \frac{\beta}{\rho} \int_\Omega |\u| (\u - \u^{i+1}_h) \cdot \v_r \, d\x + \frac{\beta}{\rho} \int_\Omega (|\u| - |\u^{i+1}_h|) \u^{i+1}_h \cdot \v_r \, d\x.
\]
Equation \eqref{upbdvit5} gives
by inserting $\pm \u$ in the second term of the right hand side and by using Property \ref{monoticity}, the following inequality
\begin{equation}\label{equat3}
\begin{array}{ll}
\medskip
 \ds  \frac{\mu}{\rho} \int_\Omega K^{-1} \textbf{z}_0  \cdot \textbf{z}_0 \, d\x +  c_m \left\|\u-\u_h^{i+1} \right\|_{L^3(\Omega)^d}^3 \le \\
 \medskip
 \ds  \frac{\mu}{\rho} \int_\Omega | K^{-1} \textbf{v}_r| \, |\textbf{z}_0|  \, d\x + \frac{\beta}{\rho} \int_\Omega |\u| \, |\u - \u^{i+1}_h| \, |\v_r| \, d\x + \frac{\beta}{\rho} \int_\Omega |\u - \u^{i+1}_h|\, |\u^{i+1}_h |\,  |\v_r| \, d\x \\
 \medskip
 \ds  + \alpha \int_\Omega |\u_h^{i+1} - \u_h^{i}| \, |\textbf{z}_0| \, d\x  +  \frac{\beta}{\rho} \int_\Omega |\u^{i+1}_h-\u^i_h| \, |\u^{i+1}_h - \u| \,  |\textbf{z}_0|   \, d\x +   \frac{\beta}{\rho} \int_\Omega |\u^{i+1}_h-\u^i_h| \, | \u| \,  |\textbf{z}_0|   \, d\x  \\
 \medskip
 + \ds \sum_{\kappa \in  \mathcal{T}_h } \big[ \int_{\kappa} |-\nabla p_h^{i+1}-\frac{\mu}{\rho}
 K^{-1} \u^{i+1}_h - \alpha (\u_h^{i+1} - \u_h^{i}) - \frac{\beta}{\rho} |\u^i_h| \u^{i+1}_h + \f_h|\, \, |\textbf{z}_0| \,d\x\, + \ds \int_{\kappa} |\f-\f_h| \, |\textbf{z}_0| \,d\x\,\big].
\end{array}
\end{equation}
By using the relation $\u - \u_h^{i+1} = \textbf{z}_0 + \v_r$ and taking into account that $W^{1,6}(\Omega) \subset L^\infty (\Omega)$, the last bound allows us to obtain the following inequality
\begin{equation}\label{equat4}
\begin{array}{ll}
\medskip
 \ds  \frac{\mu K_m}{\rho} \| \textbf{z}_0\|^2_{L^2(\Omega)} + c_m \left\|\u-\u_h^{i+1} \right\|_{L^3(\Omega)^d}^3 \le \ds \frac{\mu K_M}{\rho} \| \textbf{v}_r \|_{L^2(\Omega)} \, \| \textbf{z}_0 \|_{L^2(\Omega)} \\
 \medskip
 \ds + \frac{\beta}{\rho} \| \u \|_{L^6(\Omega)} (\| \textbf{z}_0 \|_{L^2(\Omega)} + \| \v_r \|_{L^2(\Omega)}) \| \v_r \|_{L^3(\Omega)} \\
 \medskip
 + \ds \frac{\beta}{\rho} (\| \v_r \|_{L^2(\Omega)}  + \| \textbf{z}_0 \|_{L^2(\Omega)})
 \|\u^{i+1}_h \|_{L^6(\Omega)} \, \| \v_r \|_{L^3(\Omega)} \ds + \ds \alpha \| \u_h^{i+1} - \u_h^{i} \|_{L^2(\Omega)} \, \| \textbf{z}_0 \|_{L^2(\Omega)} \\
 \medskip
 \ds + \frac{\beta}{\rho} \| \u^{i+1}_h-\u^i_h \|_{L^\infty (\Omega)}   \,
 (\| \v_r \|_{L^2(\Omega)}  + \| \textbf{z}_0 \|_{L^2(\Omega)})
  \,  \|\textbf{z}_0\|_{L^2(\Omega)}  + \frac{\beta}{\rho} \| \u^{i+1}_h-\u^i_h \|_{L^2(\Omega)}   \, \|\u \|_{L^\infty (\Omega)} \,  \|\textbf{z}_0\|_{L^2(\Omega)}\\
 \medskip
 + \ds \sum_{\kappa \in  \mathcal{T}_h } \|-\nabla p_h^{i+1}-\frac{\mu}{\rho} K^{-1} \u^{i+1}_h - \alpha (\u_h^{i+1} - \u_h^{i}) - \frac{\beta}{\rho} |\u^i_h| \u^{i+1}_h + \f_h \|_{L^2(\kappa)} \, \|\textbf{z}_0\|_{L^2(\kappa)}\\
 \medskip
 + \ds \sum_{\kappa \in  \mathcal{T}_h } \| \f-\f_h \|_{L^2(\kappa)}  \, \|\textbf{z}_0\|_{L^2(\kappa)}.
\end{array}
\end{equation}
Following Lemma \ref{bornl3}, there exists an integer $i_0$ depending on $h$ such that for all $i\ge i_0$, $\u_h^{i+1}$ is bounded in $L^6(\Omega)$. { Furthermore, we use following inverse inequality
\[
\|\u_h^{i+1} - \u_h^i \|_{L^\infty (\Omega)} \le h^{-d/2} \|\u_h^{i+1} - \u_h^i \|_{L^2 (\Omega)}
\]
and the convergence of the sequence $\u_h^i$ to $\u_h$ in $L^2(\Omega)^d$ to deduce that we can choose  $i_0$ sufficiently large such that for all $i\ge i_0$ we have $\|\u_h^{i+1} - \u_h^i \|_{L^\infty (\Omega)} \le \ds \frac{\mu K_m}{2 \beta}$.} Thus, Equation \ref{equat4} gives:
\begin{equation}\label{equat44}
\begin{array}{ll}
\medskip
 \ds  \frac{\mu K_m}{2\rho} \| \textbf{z}_0\|^2_{L^2(\Omega)} + c_m  \left\|\u-\u_h^{i+1} \right\|_{L^3(\Omega)^d}^3 \le \ds \frac{\mu K_M}{\rho} \| \textbf{v}_r \|_{L^2(\Omega)} \, \| \textbf{z}_0 \|_{L^2(\Omega)} \\
 \medskip
 \ds   +  \frac{\beta}{\rho} \| \u \|_{L^6(\Omega)} (\| \textbf{z}_0 \|_{L^2(\Omega)} + \| \v_r \|_{L^2(\Omega)}) \| \v_r \|_{L^3(\Omega)} \\
 \medskip
 + \ds \frac{\beta}{\rho} (\| \v_r \|_{L^2(\Omega)}  + \| \textbf{z}_0 \|_{L^2(\Omega)})
 \|\u^{i+1}_h \|_{L^6(\Omega)} \, \| \v_r \|_{L^3(\Omega)} \ds + \ds \alpha \| \u_h^{i+1} - \u_h^{i} \|_{L^2(\Omega)} \, \| \textbf{z}_0 \|_{L^2(\Omega)} \\
 \medskip
 \ds + \frac{\mu K_m}{2\rho} \, \| \v_r \|_{L^2(\Omega)}
  \,  \|\textbf{z}_0\|_{L^2(\Omega)}
 + \frac{\beta}{\rho} \| \u^{i+1}_h-\u^i_h \|_{L^2(\Omega)}   \, \|\u \|_{L^\infty(\Omega)}  \,  \|\textbf{z}_0 \|_{L^2(\Omega)}
 + \ds \sum_{\kappa \in  \mathcal{T}_h } \| \f-\f_h \|_{L^2(\kappa)}  \, \|\textbf{z}_0\|_{L^2(\kappa)} \\
 \medskip
 + \ds \sum_{\kappa \in  \mathcal{T}_h } \|-\nabla p_h^{i+1}- {\frac{\mu}{\rho}}
 K^{-1} \u^{i+1}_h - \alpha (\u_h^{i+1} - \u_h^{i}) -\frac{\beta}{\rho} |\u^i_h| \u^{i+1}_h + \f_h \|_{L^2(\kappa)} \, \|\textbf{z}_0\|_{L^2(\kappa)}.
\end{array}
\end{equation}
We use Lemma \ref{bornl3} and the decomposition $a b \le \ds  \frac{1}{2\varepsilon} a^2 + \frac{1}{2} \varepsilon b^2$ for all the terms containing $b=\| \textbf{z}_0 \|_{L^2(\Omega)}$  in the right hand side of Equation \eqref{equat44}, with $\varepsilon$ sufficiently small such that all the terms of $\| \textbf{z}_0 \|_{L^2(\Omega)}$ in the right hand side will be absorbed by the term $\ds \frac{\mu K_m}{2\rho} \| \textbf{z}_0\|^2_{L^2(\Omega)}$ of the left hand side of \eqref{equat44}. We then get, by using the inequality $\| \v_r \|_{L^2(\Omega)} \le |\Omega |^{1/6} \| \v_r \|_{L^3(\Omega)}$ and by taking the square root of the inequality, the bound
\begin{equation}\label{relat11}
\begin{array}{ll}
\medskip
\ds \| \textbf{z}_0\|_{L^2(\Omega)} + \left\|\u-\u_h^{i+1} \right\|_{L^3(\Omega)^d}^{3/2} \le  \ds \overline{C} \Big( \| \v_r \|_{L^3(\Omega)}  + \sum_{\kappa \in  \mathcal{T}_h } \| \f-\f_h \|_{L^2(\kappa)} + \sum_{\kappa \in  \mathcal{T}_h } \| \u_h^{i+1} - \u_h^i \|_{L^2(\kappa)} \\
 \qquad \qquad \qquad \qquad  + \ds  \sum_{\kappa \in  \mathcal{T}_h } \|-\nabla p_h^{i+1}-{\frac{\mu}{\rho}}
K^{-1} \u^{i+1}_h - \alpha (\u_h^{i+1} - \u_h^{i}) -\frac{\beta}{\rho} |\u^i_h| \u^{i+1}_h + \f_h \|_{L^2(\kappa)} \Big),
\end{array}
\end{equation}
where $\overline{C}$ is a constant depending on $(\u,p)$. 
By using 
Relations \eqref{eq:liftq2bdd1} and \eqref{relat11}, and the following inequality:
\begin{equation*}
\begin{array}{rcl}
\medskip
\| \u - \u_h^{i+1} \|_{L^2(\Omega)} &\leq& \|\textbf{z}_0\|_{L^2(\Omega)} + \|\v_r\|_{L^2(\Omega)}\\
 &\leq& \|\textbf{z}_0\|_{L^2(\Omega)} + |\Omega|^{1/6}\|\v_r\|_{L^3(\Omega)}
\end{array}
\end{equation*}
we get the desired result.
\end{proof}
{\rmq Theorem \ref{aposterioril2} gives an upper bound for the error $\u - \u_h^{i+1}= \textbf{z}_0 + \v_r$ in $L^2(\Omega)^d$ and an upper bound of $\v_r$ in $X$ which is a part of $\u-\u_h^{i+1}$. Furthermore, it gives an upper bound of $\| \u - \u_h^{i+1} \|_{L^3(\Omega)}$ but unfortunately, with the indicators to the power of $2/3$. \\
%
}

\noindent In the next theorem, we will bound the error between the gradient of the exact and numerical pressures with respect to the indicators in $L^{3/2}(\Omega)^d$. 
%
%
%
%
\begin{thm}\label{aposterioril3} Under the assumptions of Theorem \ref{aposterioril2} and we assume that $\f \in L^2(\Omega)^d$, then there exists a positive real number $i_1$ depending on $h$ such that for all $i\ge i_1$ we have the following bound between the solutions $(\u,p)$ of \eqref{V1} and $(\u_h^{i+1},p_h^{i+1})$ of $\eqref{V1hi}$:
\begin{equation}
\label{eq:uperr3}
\begin{array}{rcl}
||\nabla(p - p_h^{i+1}) ||_{L^{3/2}(\Omega)} &\leq& C\ds  \Big[ \sum_{\kappa \in  \mathcal{T}_h} \Big( \eta_{\kappa,i}^{(D_1)} + \eta_{\kappa,i}^{(D_2)} + \eta_{\kappa,i}^{(L)} \Big) \\
&& + \ds \sum_{\kappa \in  \mathcal{T}_h}  \Big( \|\textbf{f}-\textbf{f}_h\|_{L^2(\kappa)} + h_k \|b - b_h \|_{L^3(\kappa)}  + \sum_{e \in \Gamma_h^b} h_e^{\frac{1}{3}} ||g_h - g ||_{L^3(e)}  \Big)\Big],
\end{array}
\end{equation}
where $C$ is a constant depending on $(\u,p)$.
\end{thm}
\begin{proof}
Let $(\u,p)$ and $(\u_h^{i+1},p_h^{i+1})$ the solutions of \eqref{V1} and  \eqref{V1hi}. We test Equation \eqref{1.71} with
$\textbf{v}_h ={\bf 0}$ to get,
\begin{equation}\label{eqqq1}
\begin{array}{ll}
\medskip
\ds\int_{\Omega} \nabla  ( p-p_h^{i+1} ) \cdot \textbf{v}\,d\x =  \ds -  \frac{\mu}{\rho} \int_\Omega K^{-1} (\u - \u^{i+1}_h) \cdot \v \, d\x -  \frac{\beta}{\rho} \int_\Omega (|\u| \u - |\u^i_h| \u^{i+1}_h) \cdot \v \, d\x  \\
 \medskip
 \hspace{1cm}  + \ds \sum_{\kappa \in  \mathcal{T}_h } \big[ \int_{\kappa} (-\nabla p_h^{i+1}- \alpha (\u_h^i - \u_h^{i+1}) - \frac{\mu}{\rho}K^{-1} \u^{i+1}_h - \frac{\beta}{\rho} |\u^i_h| \u^{i+1}_h + \f_h) \cdot \v \,d\x\, \\
\medskip
\hspace{5cm} + \ds \int_{\kappa}
(\f-\f_h)\cdot \v\,d\x\,  + \alpha \ds \int _{\kappa} (\u_h^{i+1} - \u_h^i) \cdot \v \, d\x \big].
\end{array}
\end{equation}
By using the Cauchy-Schwartz inequality, we get after dividing the previous inequality by $||\v||_{L^3(\Omega)}$:
\begin{equation}\label{equ1}
\begin{array}{ll}
\medskip
\ds \frac{\ds \Big| \int_\Omega \nabla (p_h^{i+1} - p) \v \, d\x \Big| }{||\v ||_{L^3(\Omega)}} \le  C \ds ( || \u - \u_h^{i+1}||_{L^2(\Omega)} +  || \u_h^i - \u_h^{i+1}||_{L^2(\Omega)}) \frac{||\v||_{L^2(\Omega)}}{||\v||_{L^3(\Omega)}} \\
 \ds  + C_1 \Big( \sum_{\kappa \in \mathcal{T}_h} ||\f - \f_h||^2_{L^2(\kappa)} \Big)^{1/2}
 \frac{||\v||_{L^2(\Omega)}}{||\v||_{L^3(\Omega)}} + \frac{\beta}{\rho} \Big| \int_\Omega (|\u| \u - |\u^i_h| \u^{i+1}_h) \cdot \v \, d\x   \Big| \frac{1}{||\v||_{L^3(\Omega)}}\\
 + \ds C_2 \Big( \sum_{\kappa \in \mathcal{T}_h} || -\nabla p_h^{i+1}- \alpha (\u_h^i - \u_h^{i+1}) - \frac{\mu}{\rho}K^{-1} \u^{i+1}_h - \frac{\beta}{\rho} |\u^i_h| \u^{i+1}_h + \f_h ||^2_{L^2(\kappa)}  \Big)^{1/2} \frac{||\v||_{L^2(\Omega)}}{||\v||_{L^3(\Omega)}}.
\end{array}
\end{equation}
By using the relation $||\v||_{L^2(\Omega)} \le |\Omega |^{1/6} ||\v||_{L^3(\Omega)}$, all the term of the right hand side of the previous bound can be treated as in the previous theorem except the third one which can be bounded as following:
\begin{equation}\label{equ2}
\begin{array}{rcl}
\Big| ( |\u| \u - |\u^i_h| \u^{i+1}_h, \v )  \Big|  &\le& \Big|  ( (|\u| - |\u_h^i |) \u, \v)  \Big| +
\Big| (|\u_h^i| (\u - \u_h^{i+1}) , \v)   \Big|\\
&\le& \big(   ||\u - \u^i_h||_{L^2(\Omega)} ||\u ||_{L^6(\Omega)} + ||\u^i_h||_{L^6(\Omega} ||\u - \u_h^{i+1} ||_{L^2(\Omega)}   \big)||\v ||_{L^3(\Omega)}
\end{array}
\end{equation}
We consider Relation \eqref{equ1}. We use the following triangle inequality
\[
||\u - \u^i_h||_{L^2(\Omega)} \le ||\u - \u^{i+1}_h||_{L^2(\Omega)} + ||\u^{i+1}_h - \u^i_h||_{L^2(\Omega)},
\]
the fact that the term $||\u^i_h||_{L^6(\Omega}$ is bounded, the inf-sup condition \eqref{inf-supS1} and Theorem \ref{aposterioril2}, to get the desired error bound on the pressure given by Equation \eqref{eq:uperr3}.
\end{proof}
%
%
%
%
{\rmq The bounds \eqref{eq:uperr2} and \eqref{eq:uperr3}  constitute our {\it a posteriori} error estimates where we bound the error between the exact solution $(\u,p)$ of \eqref{V1} and the numerical solution $(\u_h^{i+1},p_h^{i+1})$ of \eqref{V1hi} with respect to the indicators $\eta_{\kappa,i}^{(L)}$, $\eta_{\kappa,i}^{(D_1)}$ and $\eta_{\kappa,i}^{(D_2)}$.
 But to get the bounds of the indicators which are the subject of the next subsection (Section \ref{optimall}), we need to add the following theorem where we add a supplementary bound giving an error bound of the exact and numerical solutions.
}
{\thm Under the assumptions of Lemma \ref{bornl3},  there exists an integer $i_0$ depending on $h$ such that for all $i\ge i_0$,  the solutions $(\u,p)$ of \eqref{V1} and $(\u_h^{i+1},p_h^{i+1})$ of $\eqref{V1hi}$ verify the following error inequalities:
\begin{equation}
\label{eq:uperr211}
\begin{array}{ll}
\|\ds \frac{\beta}{\rho}(|\u|\u-|\u_h^i|\u_h^{i+1}) + \nabla (p - p_h^{i+1})\|_{L^2(\Omega)} \leq C\ds  \Big[ \sum_{\kappa \in  \mathcal{T}_h} \Big( \eta_{\kappa,i}^{(D_1)} + \eta_{\kappa,i}^{(D_2)} + \eta_{\kappa,i}^{(L)} \Big) \\
 \hspace{3cm} \ds + \sum_{\kappa \in  \mathcal{T}_h}  \Big( \|\textbf{f}-\textbf{f}_h\|_{L^2(\kappa)} + h_k \|b - b_h \|_{L^3(\kappa)}  + \sum_{e \in \Gamma_h^b} h_e^{\frac{1}{3}} ||g_h - g ||_{L^3(e)}  \Big)\Big],
\end{array}
\end{equation}
where $C$ is a constant depending on the exact solution $(\u,p)$ of \eqref{V1}.
}
\begin{proof}
Let $\u \in W^{1,6}(\Omega)^d$. Then, Equation \eqref{E1} allows us to get that the pressure is such that $\nabla p \in L^2(\Omega)^d$. Thus, the velocity error equation \eqref{1.71} is valid for all $\v \in L^2(\Omega)^d$ and can be written as
\begin{equation}\label{eq:velocerr3}
\begin{array}{ll}
\medskip
\ds \int_\Omega
 \big( \frac{\beta}{\rho}  (|\u| \u - |\u^{i}_h| \u^{i+1}_h) + \nabla  ( p-p_h^{i+1} )   \big) \cdot \v \, d\x  =  \ds -   \frac{\mu}{\rho} \int_\Omega K^{-1} (\u - \u^{i+1}_h) \cdot \v \, d\x + \alpha \int_\Omega (\u_h^{i+1} - \u_h^{i}) \cdot \v \, d\x \\
 \medskip
 + \ds \sum_{\kappa \in \mathcal{T}_h } \big[ \int_{\kappa} (-\nabla p_h^{i+1}-{{
 \frac{\mu}{\rho}K^{-1} \u_h^{i+1}}}-
 \alpha (\u_h^{i+1}-\u_h^i)-\frac{\beta}{\rho}|\u_h^i|\u_{h}^{i+1}+\f_h)\cdot (\v-\v_h)\,d\x + \int_{\kappa}(\f-
 \f_h) \cdot \v \,d\x
 \big].
\end{array}
\end{equation}
By taking $\v_h = \0$ and $\v = \ds \big( \frac{\beta}{\rho}  (|\u| \u - |\u^{i}_h| \u^{i+1}_h) + \nabla  ( p-p_h^{i+1} ) \big)$,  applying the Cauchy-Schwartz and simplifying by $||\v||_{L^2(\Omega)}$, we get the result after using Theorem \ref{aposterioril2}.
\end{proof}
{\rmq Finally, the bounds \eqref{eq:uperr2}, \eqref{eq:uperr3} and \eqref{eq:uperr211} constitute our {\it a posteriori} error estimates.
}
\subsubsection{Bounds of the indicators}\label{optimall}
\noindent In order to establish the efficiency of the {\it a posteriori} error estimates, we recall the following properties (see R. Verf\"urth,\cite{Verfurth2013}, Chapter 1). For an element $\kappa$ of $\mathcal{T}_h$, we consider the bubble function $\psi_\kappa$ (resp. $\psi_e$ for the face $e$) which is equal to the product of the $d+1$ barycentric coordinates associated with the vertices of $\kappa$ (resp. of the $d$ barycentric coordinates associated with the vertices of $e$). We also consider a lifting operator ${\mathcal{L}}_{e}$ defined on polynomials on $e$ vanishing on $\partial e$ into polynomials on the at most two elements $\kappa$ containing $e$ and vanishing on $\partial \kappa \setminus e $, which is constructed by affine transformation from a fixed operator on the reference element.
\begin{propri}\label{psi}
 Denoting by $Pr(\kappa)$ the space of polynomials of degree smaller than $r$ on $\kappa$. The following properties hold:
\begin{equation}
\forall v \in P_r(\kappa), \qquad
\begin{cases}
c ||v||_{0,\kappa} \le ||v \psi^{1/2}_{\kappa} ||_{0,\kappa}
\le c' ||v||_{0,\kappa}, &\\
|v|_{1,\kappa} \le c h_{\kappa}^{-1} ||v ||_{0,\kappa}.&
\end{cases}
\end{equation}
\end{propri}
\begin{propri} \label{psii} Denoting by $Pr(e)$ the space of polynomials of degree smaller than $r$ on $e$, we have
$$\forall\; v  \in P_r(e),\qquad
c\Vert v \Vert_{0,e}\leq \Vert v\psi_{e}^{1/2}
\Vert_{0,e}\leq c'\Vert  v \Vert_{0,e},$$
and, for all polynomials $v$ in $Pr(e)$ vanishing on $\partial e$, if $\kappa$ is an element which contains $e$,
$$ \Vert {\mathcal{L}}_{e}v \Vert_{0,\kappa}+h_{e}\mid
{\mathcal{L}}_{e}v \mid_{1,\kappa}\leq ch^{1/2}_{e}\Vert  v
\Vert_{0,e}.$$
\end{propri}
We have the following bounds of the indicators:
\begin{thm}
Let $d=3$,  $(\u,p)$  and $(\u_h^{i+1},p_h^{i+1})$ the solutions of \eqref{V1} and  of $\eqref{V1hi}$. We have the following bounds of the indicators: for each element $\kappa \in \mathcal{T}_h$,
\begin{equation}\label{lowerbounds1}
\eta^{(L)}_{\kappa,i}  \le ||\u - \u_h^i ||_{L^2(\kappa)} + ||\u - \u_h^{i+1} ||_{L^2(\kappa)},
\end{equation}
and
\begin{equation}\label{lowerbounds3}
\begin{array}{rcl}
\medskip
\eta^{(D_2)}_{\kappa,i}  &\le&
C\ds  \Big( \| \v_r\|_{L^3(w_\kappa)} + h_k \|b - b_h \|_{L^3(w_\kappa)}  + \sum_{e\in \partial \kappa}
h_e^{\frac{1}{3}}
||g_h - g ||_{L^3(e)}\Big),
\end{array}
\end{equation}
where $C$ is a constant independent of the mesh step but depends on the exact solution $(\u,p)$.
\end{thm}
\begin{proof}
The bound \eqref{lowerbounds1} is a simple consequence of the definition of $\eta_{\kappa,i}^{(L)}$ and the triangle inequality.\\
\noindent In order to prove \eqref{lowerbounds3}, we consider first Equation \eqref{1.72} with $q_h=0$ and
\[
q=q_\kappa = \left\{
\begin{array}{lcl}
b_h \psi_\kappa & \mbox{on } \kappa, \\
0 & \mbox{on } \Omega \backslash \kappa,
\end{array}
\right.
\]
where $\psi_\kappa$ is the bubble functions on a given element $\kappa \in \mathcal{T}_h$. We obtain by using Relation \eqref{perpond} the following equation:
\begin{equation}
\ds \int_{\kappa} b_h^2 \psi_\kappa  \,d\x = - \ds \int_{\kappa} \nabla (b_h \psi_\kappa)  \cdot
\v_r\,d\x
- \int_{\kappa} (b_h \psi_\kappa) (b - b_h) \,d\x.\,
\end{equation}
Then we use Property \ref{psi}, the Cauchy-Schwartz inequality and the relation $|| \v ||_{L^2(\kappa)} \le |\kappa |^{1/6} ||\v ||_{L^3(\kappa)}$, and by multiplying by $h_\kappa$ to get:
\begin{equation}\label{term1}
\begin{array}{rcl}
\medskip
h_\kappa ||b_h ||_{L^2(\kappa)} &\le& c h_\kappa ( h_\kappa^{-1} ||\v_r ||_{L^2(\kappa)}  + ||b - b_h ||_{L^2(\kappa)})\\
&\le& c_1 ( h_\kappa^{\frac{1}{2}} ||\v_r ||_{L^3(\kappa)}  + h^{\frac{3}{2}}_\kappa ||b - b_h ||_{L^3(\kappa)}).
\end{array}
\end{equation}
Then we get by using the inverse inequality \eqref{eq:inversin} with $p=3$,
\begin{equation}\label{aux0}
h_\kappa ||b_h ||_{L^3(\kappa)} \le c_2 ( ||\v_r ||_{L^3(\kappa)}  + h_\kappa ||b - b_h ||_{L^3(\kappa)}),
\end{equation}
which is the part of the indicator $\eta_{\kappa,i}^{(D_2)}$ corresponding to $b_h$. \\
\noindent Again, we consider Equation \eqref{1.72} with $q_h=0$ and
\begin{equation*}
q=q_{e}=
\left \{
\begin{array}{lcl}
\mathcal{L}_{e,\kappa } \big(
 \phi_{h,1}^e \psi_{e}  \big)& \hspace{-0.5cm} \mbox{on } \{ \kappa, \kappa' \}, \\
0 & \hspace{-0.3cm}\mbox{ on } \Omega \backslash  (\kappa \cup \kappa') , \\
\end{array}
 \right.
\end{equation*}
\noindent where $\psi_e$ is the bubble function of $e$ and $\kappa'$ denotes the other element of $\mathcal{T}_h$ that share  $e$ with $\kappa$. We get the following equation:
\begin{equation*}
\ds \int _{e} \phi_{h,1}^e q\,d\s
= \ds - \int_{\kappa \cup \kappa'} q (b - b_h) \,d\x\, - \int_{\kappa \cup \kappa'} b_h q  \,d\x
- \int _{e} (g_h - g ) q \,d\s - \int_{\kappa \cup \kappa'} \nabla q \cdot
\v_r\,d\x.
\end{equation*}
Properties \ref{psi} and \ref{psii} allow us to get the following bound:
\begin{equation*}
||\phi_{h,1}^e||_{L^2(e)} \le C \big( h_e^{\frac{1}{2}} || b - b_h ||_{L^2(\kappa \cup \kappa')} +  h_e^{\frac{1}{2}} ||b_h ||_{L^2(\kappa \cup \kappa')}  + ||g - g_h ||_{L^2(e)}  +  h_e^{-\frac{1}{2}} ||\v_r||_{L^2(\kappa \cup \kappa')} \big).
\end{equation*}
By using again the inverse inequality \ref{eq:inversin} and the relation $\| \v_h \|_{L^2(\kappa)} \le |\kappa |^{1/6} \| \v_h \|_{L^3(\kappa)} $, we obtain the bound:
\begin{equation}\label{aux00}
h_e^{\frac{1}{3}} ||\phi_{h,1}^e||_{L^3(e)} \le C_1 \big( h_e || b - b_h ||_{L^3(\kappa \cup \kappa')} +  h_e ||b_h ||_{L^3(\kappa \cup \kappa')}  +
h_e^{\frac{1}{3}} ||g - g_h ||_{L^3(e)}  +  ||\v_r||_{L^3(\kappa \cup \kappa')} \big).
\end{equation}
Hence, we bound the part of $\eta_{\kappa,i}^{(D_2)}$ corresponding to $\phi_{h,1}^e$. Relations \eqref{aux0} and \eqref{aux00} give Relation \eqref{lowerbounds3}.
\end{proof}
\begin{thm}
Let $d=3$ and let the mesh satisfy \eqref{eq:reg}. Under the assumptions of Lemma \ref{bornl3}, we have
the following bound: for each element $\kappa \in \mathcal{T}_h$,
\begin{equation}\label{lowerbounds2}
\begin{array}{rcl}
\medskip
\eta^{(D_1)}_{\kappa,i}  &\le&
C\ds  \Big( \eta_{\kappa,i}^{(L)}  + \| \u - \u_h^{i+1} \|_{L^2(w_\kappa)} + \|  \frac{\beta}{\rho} 
\big( |\u| \u - |\u_h^{i+1}| \u_h^{i+1}  \big) + \nabla ( p-p_h^{i+1} )\|_{L^2(\kappa)}  \\
&& \hskip 1cm \ds + \| K^{-1} - K_h^{-1} \|_{L^{3}(w_\kappa)} + \|\textbf{f}-\textbf{f}_h\|_{L^2(w_\kappa)} + h_k \|b - b_h \|_{L^3(w_\kappa)}  + \sum_{e\in \partial \kappa} h_e^{\frac{1}{3}} ||g_h - g ||_{L^3(e)}\Big),
\end{array}
\end{equation}
where $C$ is a constant independent of the mesh step but depends on the exact solution $(\u,p)$ and $K_h^{-1}$ is an approximation of $K^{-1}$ which is a constant tensor in each triangle.
\end{thm}
\begin{proof}
Let us now prove Relation \eqref{lowerbounds2}. We consider Equation \eqref{1.71} with $\v_h=0$ and
\begin{equation*} \v=\v_{\kappa}=
\left \{
\begin{array}{lcl}
\ds \big( -\nabla p_h^{i+1}- \alpha 
(\u_h^{i+1} - \u_h^{i}) -
\frac{\mu}{\rho}
K_h^{-1} \u^{i+1}_h - \frac{\beta}{\rho} |\u^i_h| \u^{i+1}_h + \f_h \big)\psi_{\kappa} && \hspace{-0.cm} \mbox{ on } \kappa, \\
0 && \hspace{-0.cm}\mbox{ on } \Omega \backslash \kappa, \\
\end{array}
 \right.
\end{equation*}
where $K_h^{-1}$ is an approximation of $K^{-1}$ which is a constant tensor in each triangle.
We obtain the following equation:
\begin{equation*}
\begin{array}{ll}
\medskip
 \ds  \int_\kappa | (-\nabla p_h^{i+1}- \alpha 
 (\u_h^{i+1} - \u_h^{i}) -
 \frac{\mu}{\rho}
 K_h^{-1} \u^{i+1}_h - \frac{\beta}{\rho} |\u^i_h| \u^{i+1}_h + \f_h) \psi_\kappa^{1/2} |^2 d\x = \\
 \medskip
 \hspace{2cm} \ds \frac{\mu}{\rho}
 \int_\kappa   (K^{-1} - K_h^{-1}) \u^{i+1}_h \cdot \v \, d\x +
 \frac{\mu}{\rho} \int_\kappa K^{-1} (\u - \u^{i+1}_h) \cdot \v \, d\x +
 \frac{\beta}{\rho} \int_\kappa (|\u| \u - |\u^i_h| \u^{i+1}_h) \cdot \v \, d\x \\
 \medskip
 \hspace{2cm}  \ds + \ds\int_{\kappa} \nabla  ( p-p_h^{i+1} ) \cdot \textbf{v}\,d\x - \ds \int_{\kappa}
(\f-\f_h)\cdot \v\,d\x\,  - \alpha \ds \int _{K} (\u_h^{i+1} - \u_h^i) \cdot \v \, d\x,
\end{array}
\end{equation*}
and then by using Lemma \ref{bornl3} and Properties \ref{psi} and \ref{psii}, we get  the following bound:
\begin{equation}
\begin{array}{ll}
\medskip
 \ds  \| -\nabla p_h^{i+1}- 
 \alpha (\u_h^{i+1} - \u_h^{i}) -  \frac{\mu}{\rho}
 K^{-1}_h \u^{i+1}_h - \frac{\beta}{\rho} |\u^i_h| \u^{i+1}_h + \f_h
 \|_{L^2(\kappa)} \le \\
 \medskip
 \hspace{2cm} \ds C \Big( \| K^{-1} - K_h^{-1} \|_{L^{3} (\kappa)}   +
 \| \u - \u^{i+1}_h \|_{L^2(\kappa)} + \| \frac{\beta}{\rho} \big(|\u| \u - |\u^i_h| \u^{i+1}_h \big) + \nabla (p-p_h^{i+1}) \|_{L^2(\kappa)}\\
 \hspace{3cm} \ds + \ds \|\f-\f_h \|_{L^2(\kappa)} + \alpha \| \u_h^{i+1} - \u_h^i \|_{L^2(\kappa)} \Big).
\end{array}
\end{equation}
Thus we get the result by using the following triangle inequality and Lemma \ref{bornl3} :
\begin{equation}\label{aux11}
\begin{array}{ll}
\medskip
 \ds  \| -\nabla p_h^{i+1}- 
 \alpha (\u_h^{i+1} - \u_h^{i}) -  \frac{\mu}{\rho}
 K^{-1} \u^{i+1}_h - \frac{\beta}{\rho} |\u^i_h| \u^{i+1}_h + \f_h
 \|_{L^2(\kappa)} \le \\
 \medskip
 \hspace{.5cm} \ds   \| -\nabla p_h^{i+1}- 
 \alpha (\u_h^{i+1} - \u_h^{i}) -  \frac{\mu}{\rho}
 K^{-1}_h \u^{i+1}_h - \frac{\beta}{\rho} |\u^i_h| \u^{i+1}_h + \f_h
 \|_{L^2(\kappa)} + \frac{\mu}{\rho} \| K^{-1}  - K^{-1}_h \|_{L^3(\kappa)} \|  \u^{i+1}_h \|_{L^6(\kappa)}. 
\end{array}
\end{equation}
\end{proof}

\section{Numerical simulation}
We validate the theory developed here by showing numerical simulations using Freefem++ (see \cite{hecht}). \\
We consider the iterative scheme \eqref{V1hi}.
For the stopping criterion given later, we define the  iterative error  
$$Err_L=\ds \Big( \frac{||\u_h^{i+1}-\u_h^{i}||_{L^3(\Omega)} + ||\nabla(p_h^{i+1}-p_h^{i})||_{L^{3/2}(\Omega)}}{||\u_h^{i+1}||_{L^3(\Omega)} + ||\nabla p_h^{i+1}||_{L^{3/2}(\Omega)}} \Big).$$
In the definition of $Err_L$, we consider the iterative error $||\u_h^{i+1}-\u_h^{i}||_{L^3(\Omega)}$ in the natural space of the velocity $L^3(\Omega)$ although in the definition of $ \eta_{\kappa,i}^{(L)}$, we used the error in $L^2(\kappa)$.
\subsection{First test case}
In this section, the domain $\Om$ is the square $\Omega=]0,1[^2$ and all computations start on a uniform initial triangular mesh obtained by dividing $\Omega$ into $N^2$ equal squares, each one subdivided into $2$ triangles, so that the initial triangulation consists of $2N^2$ triangles.
\noindent The theory is tested by applying the numerical scheme \eqref{V1hi} to  the  exact solution
$(\u,p,T)=({\bf curl}\, \psi,p,T)$
where $\psi$ and $p$ are given by
\begin{equation}
\label{eq:1.78}
\ds \psi(x,y)=e^{-  \gamma  ((x-0.5)^2+(y-0.5)^2)}
\end{equation}
and
\begin{equation}
\label{eq:1.79}
\ds p(x,y)=x*(x-2./3.)*y*(y-2./3.),
\end{equation}
 with the choice  $K=I$, $\mu=\rho=1$ and $\gamma  =50$. Here we have $\u.\n=0$ and $b=\div(\u)=0$.\\
 
We begin by testing the dependency of the convergence of the iterative scheme \eqref{V1hi} with respect to $\alpha$.
We consider $N=60$ and for each $\alpha$, we stop the algorithm \eqref{V1hi} when the error $Err_L <1e^{-5}$. \\
%
To discribe the convergence of Algorithm \eqref{V1hi}, we consider also the error
\[
Err=\ds\Big( \frac{||\u_h^{i}-\u||_{L^3(\Omega)} + ||\nabla(p_h^{i}-p)||_{L^{3/2}(\Omega)}}{||\u||_{L^3(\Omega)} + ||\nabla p||_{L^{3/2}(\Omega)}}\Big).
\]
We mention that in the definition of $Err$, we considered $\|\u_h^i - \u\|_{L^3(\Omega)}$ despite that in Theorem \ref{aposterioril2} it figures to the power $3/2$.\\

Tables \ref{tab111*} and \ref{tab112*} show, for $\beta=1$ and $\beta=10$, the error $Err$ and the number of iterations $Nbr$ which describe the convergence of Algorithm \eqref{V1hi} with respect to $\alpha$ for $\u_h^0=\0$. We remark that the best convergence is obtained for $\alpha_{min}=2.3$ when $\beta=1$ and for $\alpha_{min}=12$ when $\beta=10$.\\
\begin{table}[h!]
\begin{tabular}{|l|l|l|l|l|l|l|l|l|l|l|l|l|l|l|l|l|}
\hline
\bf $\alpha$  & \bf 0.001 & \bf 0.01 & \bf 0.1 & \bf 1&\bf 1.4 &\bf 1.9& \bf 2.1& \bf 2.3& \bf 2.5& \bf 2.7& \bf 3 & \bf 3.7 &\bf 5 &\bf 10 & \bf 100 & \bf 1000  \\
\hline
 \bf Nbr  & 75 & 74 & 63 &25&20&16&15&14&15&16&17&20&24 & 41 & 231 & 1134  \\
\hline
\end{tabular}
\caption{
Number of iterations $Nbr$  for each $\alpha$. ($\beta=1$  and $\u_h^0=\0$)). In all these cases, $Err=-0.939$ (in logarithmic scale).  }\label{tab111*}
%
%
\begin{tabular}{|l|l|l|l|l|l|l|l|l|l|l|l|l|l|l|l|l|l|l|}
\hline
\bf $\alpha$ & \bf 0.001 & \bf 0.01 & \bf 0.1 & \bf 1 &\bf6& \bf 8&\bf 10 &\bf 11 & \bf 12& \bf 13&\bf 14&\bf 15&\bf 18&\bf21 &\bf 35 &\bf 100 & \bf 1000 \\
\hline
 \bf Nbr  & 693 & 681 & 579 & 232 &55&42& 34 &32&30&31&32&34&38&42&59& 116 & 465  \\
\hline
\end{tabular}
\caption{Number of iterations $Nbr$  for each $\alpha$. ($\beta=10$ and $\u_h^0=\0$). In all these cases, $Err=-0.475$ (in logarithmic scale).}\label{tab112*}
\end{table}

To go far with our numerical investigations, we test Algorithm \eqref{V1hi} where the initial guess $\u_h^0$ is calculated by using the Darcy's problem (which corresponds to $\beta=\alpha=0$).  Tables \ref{tab111} and \ref{tab112} show, for $\beta=1$ and $\beta=10$, the error $Err$ and the number of iterations $Nbr$ with respect to $\alpha$. We remark that here also the best convergence is obtained for $\alpha_{min}=1.4$ when $\beta=1$ and for $\alpha_{min}=14$ when $\beta=10$. In this case of the initial guess, we remark that the number of the iterations is slightly smaller than that obtained for $\u_h^0=\0$. Thus, in the following, all the numerical investigations will be performed with the initial guess  calculated by using the Darcy's problem.\\

\begin{table}[h!]
\begin{tabular}{|l|l|l|l|l|l|l|l|l|l|l|l|l|l|l|l|l|l|l|}
\hline
\bf $\alpha$  & \bf 0.001 & \bf 0.01 & \bf 0.1 & \bf 1 & \bf1.4& \bf 1.9& \bf 2.3&\bf 2.6&\bf 2.9&\bf 3.3&\bf 3.7& \bf 4& \bf 5& \bf 6&\bf 10 & \bf 100 & \bf 1000  \\
\hline
 \bf Nbr  & 74 & 73 & 62 & 25&20&16&13&12&13&14&15&16&18&21 & 30 & 174 & 904  \\
\hline 
\end{tabular}
\caption{Number of iterations $Nbr$  for each $\alpha$. ($\beta=1$). $\u_h^0$ is calculated by using the Darcy's problem. In all these cases, $Err=-0.939$ (in logarithmic scale).}\label{tab111}
%
%
\begin{tabular}{|l|l|l|l|l|l|l|l|l|l|l|l|l|l|l|l|l|l|l|l|}
\hline
\bf $\alpha$ & \bf 0.001 & \bf 0.01 & \bf 0.1& \bf 1 &\bf 6 & \bf 10&\bf 11 &\bf 12 & \bf 13& \bf 14 &\bf 15&\bf 16&\bf 18&\bf 21& \bf 28 & \bf 100 & \bf 1000 \\
\hline
 \bf Nbr  & 692 & 680 & 578 & 232&55 & 34 &31&29&27&26&27&28&30&33&40 & 92 & 458  \\
\hline
\end{tabular}
\caption{Number of iterations $Nbr$  for each $\alpha$. ($\beta=10$). $\u_h^0$ is calculated by using the Darcy's problem. In all these cases, $Err=-0.475$ (in logarithmic scale).}\label{tab112}
\end{table}
Furthermore, Table 5 shows the dependancy of $\alpha_{min}$ whith respect to $h$ for $\beta=100$. We remark that $\alpha_{min}$ increases when $h$ decreases, which is consistent with the results of the Theorems \ref{boundu1} and \ref{converg1} (see Remarks \ref{zero} and \ref{zero2}).\\
\begin{table} [h!]
\begin{tabular}{|l|l|l|l|l|l|l|l|}
\hline
\bf $h$ & \bf 0.1414 &\bf 0.0708 & \bf 0.0353 & \bf 0.0283 & \bf 0.0177 & \bf 0.007889 & \bf 0.00543\\
\hline 
\bf $\alpha_{min}$ & 66& 68.5 & 80.5 & 82.25 &82.75 & 85.5 & 87 \\
\hline
\end{tabular}
\caption{$\alpha_{min}$ with respect to $h$. ($\beta=100$ and $\u_h^0=0$).}
\end{table}
\\
In the following of this section, we will show numerical investigations corresponding to the {\it a posteriori} error estimate. We take $\beta=10$, $\alpha=10$ and $N=10$ for the initial mesh.\\
On a given mesh and for the numerical calculation, it is convenient to compute the following indicators:
\begin{equation}\nonumber
 \eta_{i}^{(D)}=\big(\sum _{K \in
 \mathcal{T}_h }(\eta_{K,i}^{(D_1)})^2 + (\eta_{K,i}^{(D_2)})^2
\big)^{\frac{1}{2}}
\end{equation}
and
\begin{equation}\nonumber
 \eta_{i}^{(L)}=\big(\sum _{K \in
 \mathcal{T}_h }((\eta_{K,i}^{(L)})^2
\big)^{\frac{1}{2}}.
\end{equation}
The iterations are stopped following the criteria
\begin{equation}\label{stoppingg}
\eta_{i}^{(L)}\leq \tilde{\gamma} \eta_{i}^{(D)},
\end{equation}
where $\tilde{\gamma}=0.001$. For the study of the dependence of the stopping criteria $\eqref{stoppingg}$ with $\tilde{\gamma}$, we refer to \cite{LAM11} and
\cite{ERN11} where the authors introduce this new stopping criterion.\\
%
For the adaptive mesh (refinement and coarsening), we use routines in FreeFem++. The indicators \eqref{indicatorss} are used for mesh adaptation by the adapted mesh algorithm  introduced in \cite{BDMS}. \\

In Figure \ref{figure1.1}, we present the evolution of the mesh
during the iterations (initial, second and fourth refinement levels). We notice
that the mesh is concentrated in the region where the solution needs to be well described.
\begin{figure}[h!]
\hskip-.1cm
\begin{subfigure}[b]{0.35\textwidth}
\centering
\includegraphics[width=5cm]{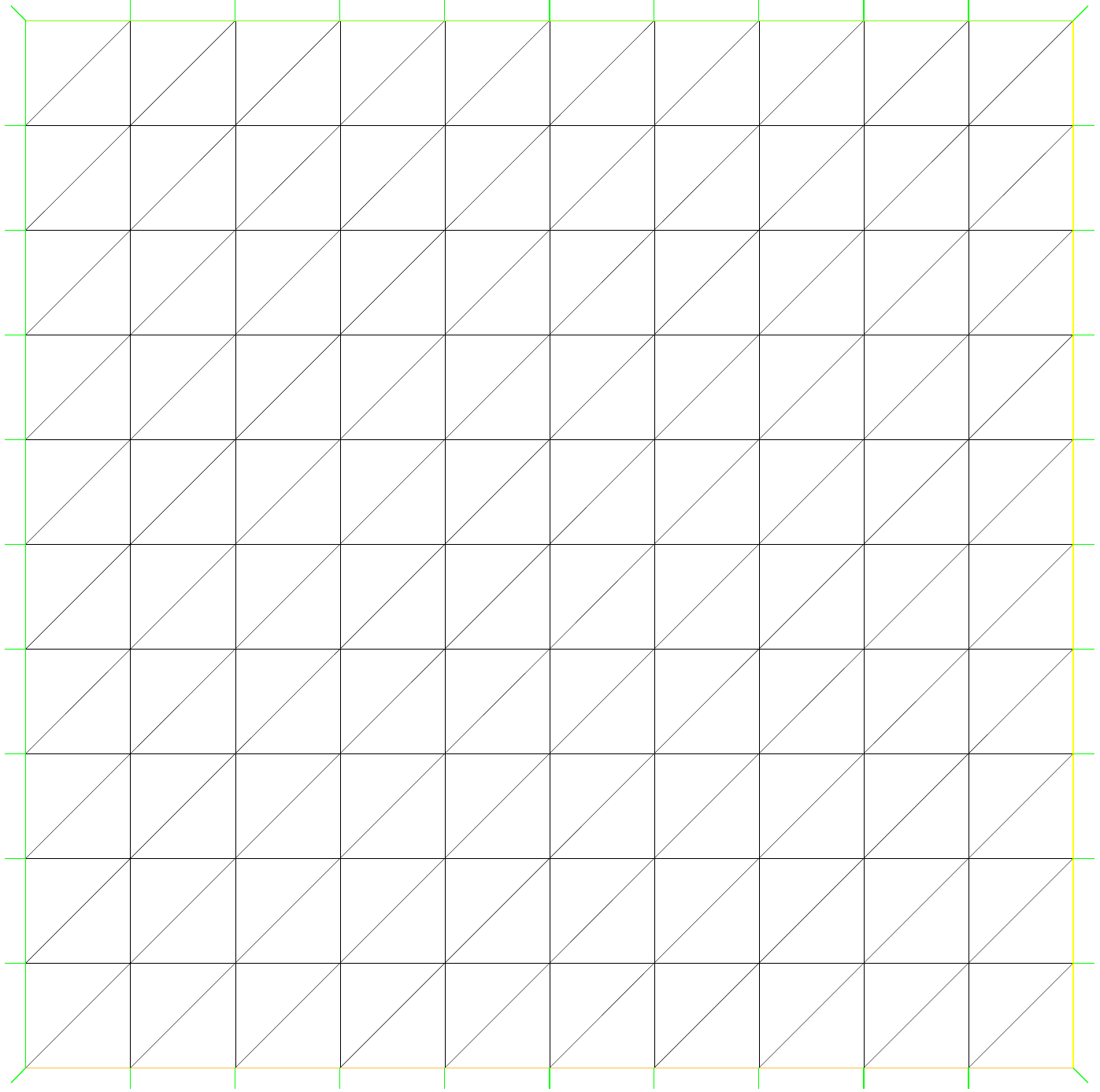}
\end{subfigure}
\hskip-.5cm
\begin{subfigure}[b]{0.35\textwidth}
\centering
\includegraphics[width=5cm]{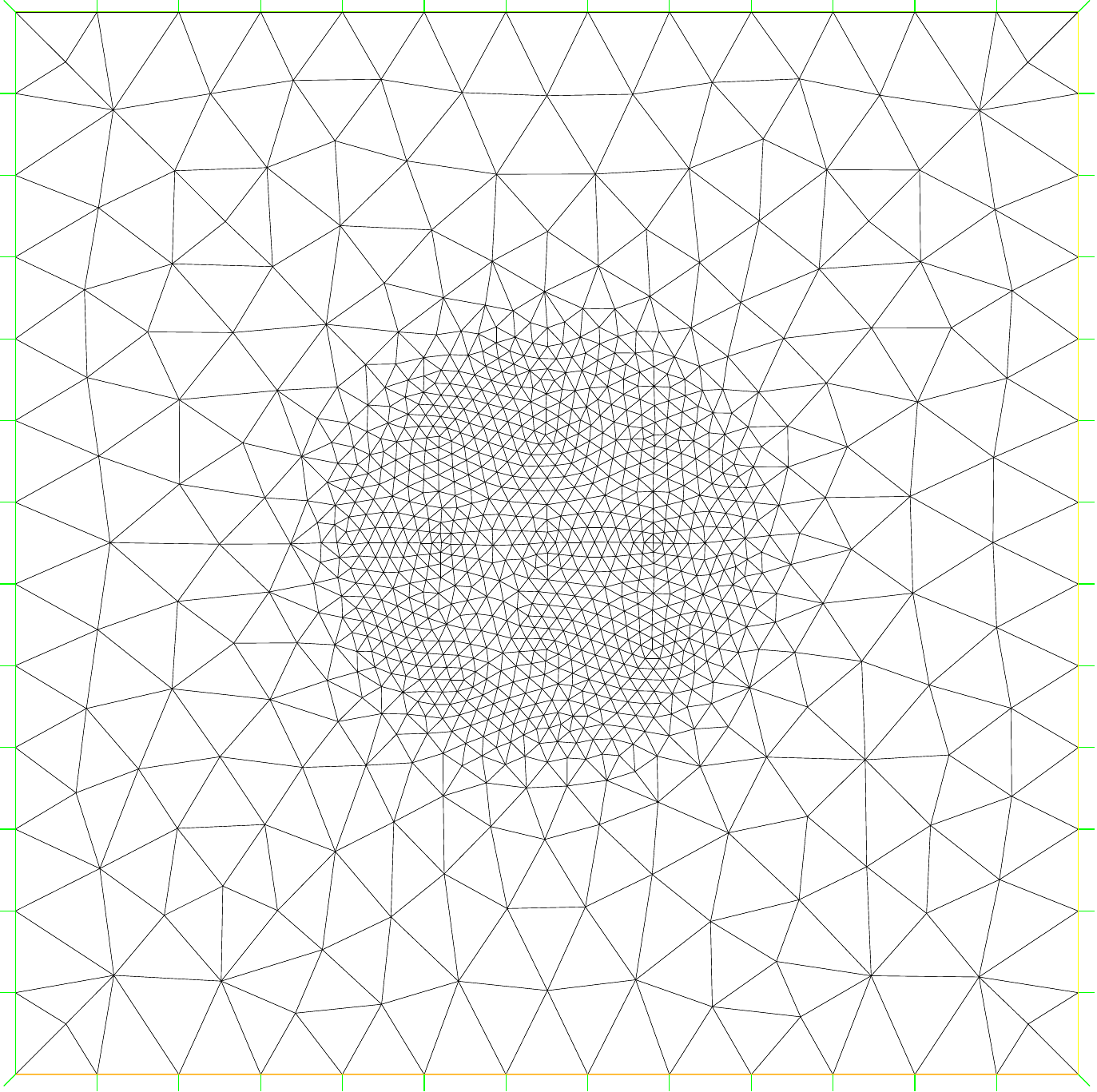}
\end{subfigure}
%
\hskip-.5cm
\begin{subfigure}[b]{0.35\textwidth}
\centering
\includegraphics[width=5cm]{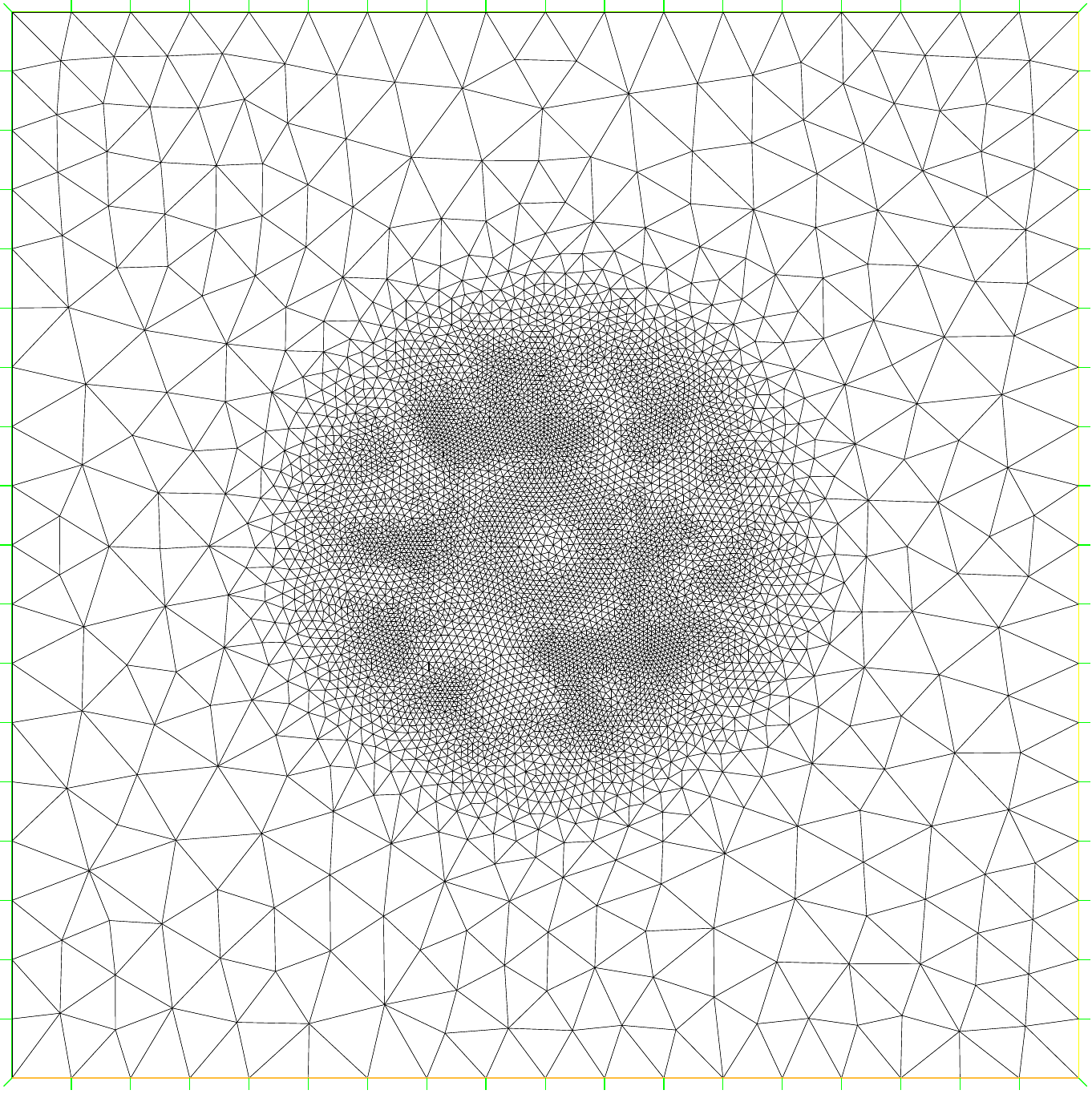}
\end{subfigure}
\caption{Evolution of the mesh during the refinement levels (initial, second and fourth).}\label{figure1.1}
\end{figure}

Next, we plot and study the error curves between the exact and numerical solutions corresponding to uniform and adaptive method. \\

Figure \textsc{\ref{figure1.2}} plots a comparison of the global error curves $Err$ versus the total number of  vertices in logarithmic scales for the uniform and adapt methods; global in the sense that they  depict the sum of the velocity and pressure errors. We notice that the errors of the adaptive mesh method are much smaller than that obtained with the uniform method, hence the efficiency of this method.
\begin{figure}[h!]
\centering
\includegraphics[width=9.5cm]{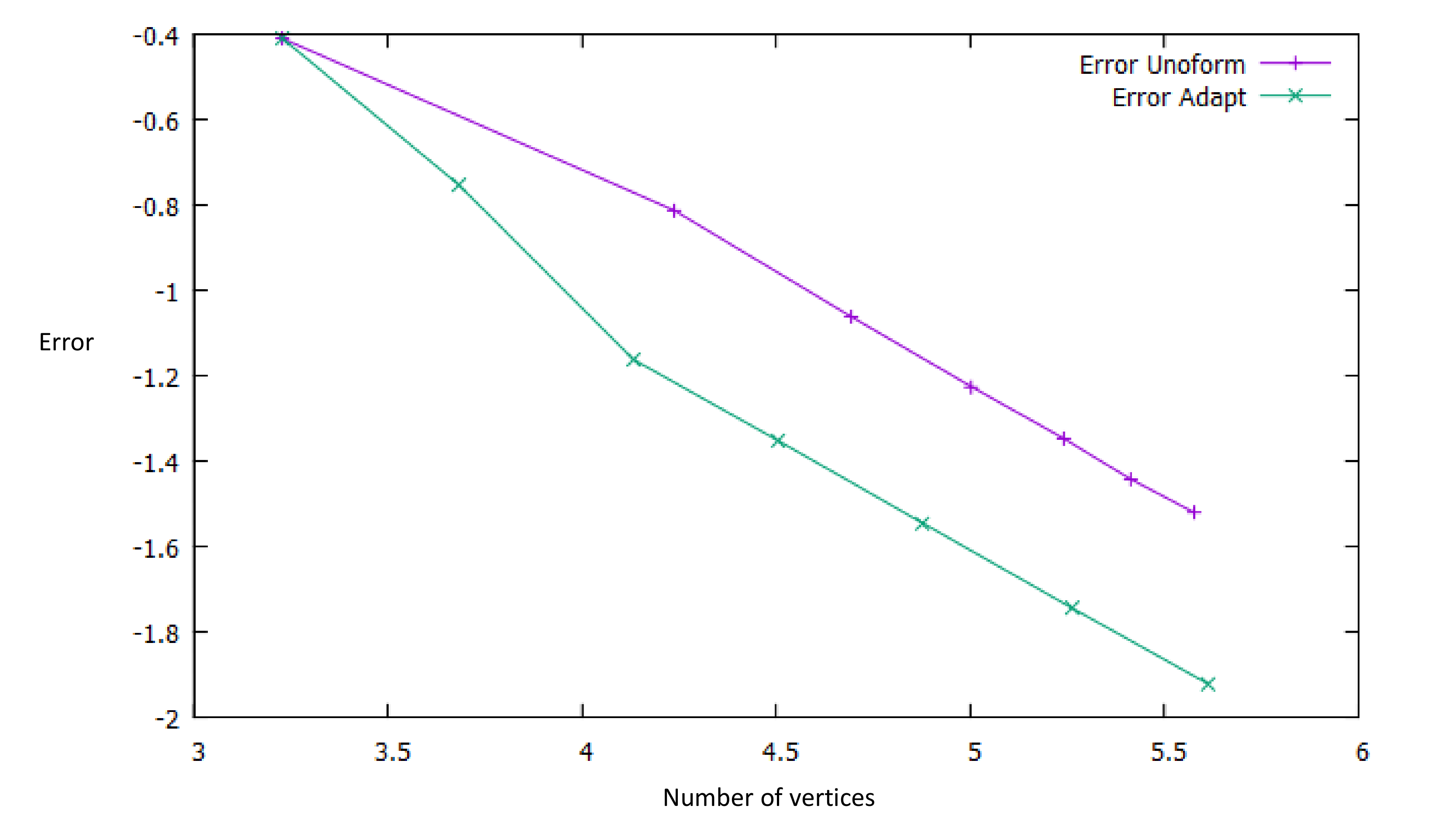}
\vskip -.4cm
\caption{Comparison of the errors $Err$ with respect to the total number of vertices in logarithmic scale.}\label{figure1.2}
\end{figure}

In table \ref{tableconst1}, we present the effectivity index defined as
\[
EI = \ds \frac{\ds \eta_i^{(L)} +  \eta_i^{(D)}}{\ds
 \|\textbf{u}-\textbf{u}_h^{i+1}\|_{L^3(\Omega)} + ||\nabla(p - p_h^{i+1|}) ||_{L^{3/2}(\Omega)} }
\]
with respect to the number of vertices during the refinement levels. This effectivity index is calculated on each mesh level after the convergence on the iterations i by using the stopping criteria \eqref{stoppingg}. Table \ref{tableconst1} shows that it is between $38.51$ and $25.47$.
\begin{table}[h!]
\begin{tabular} {|l||*{18}{c|}}
 \hline Refinement Level &  initial  & first  & second  & third  & fourth & fifth & sixth \\
 \hline
 Number of vertices
 &  121  &  313  &  973  &  2638 & 6197 & 15358 &  37703  \\
 \hline Effectivity index  & 38.51 & 32.53 & 30.54 & 27.91 & 28.70 & 27.15 & 25.47  \\
\hline
\end{tabular}
\vspace{.2cm} \caption{ EI with respect to the refinement levels.}
\label{tableconst1}
\end{table}
\subsection{Second test case} In this case, we consider a more complicated geometry \ref{geomecorners} presenting reentrant corners to show the efficiency of the adaptive method proposed in this work. Furthermore, we take $\mu=\rho=1,\alpha=10, \beta=10, \u.\n=0, b=\div(\u)=0$, $\f=(f,0)$ where 
$$
f=\left\{
\begin{array}{rcl}
0     && \quad \mbox{if } y >1,\\ 
-2     && \quad \mbox{if } y <=1, 
\end{array}
\right.
$$
and $K$ such that 
$$
K^{-1}= \left(
\begin{array}{lcl}
2+\sin(\pi x) \sin(\pi y) && \quad 0.2 x\\
0.2 x && \quad 3+\sin(\pi x) \sin(\pi y)
\end{array}
\right)
$$
for all the numerical simulations of this section. 
\begin{figure}[!ht]
\vspace{-.3cm}
\begin{center}
\includegraphics[width=6.cm]{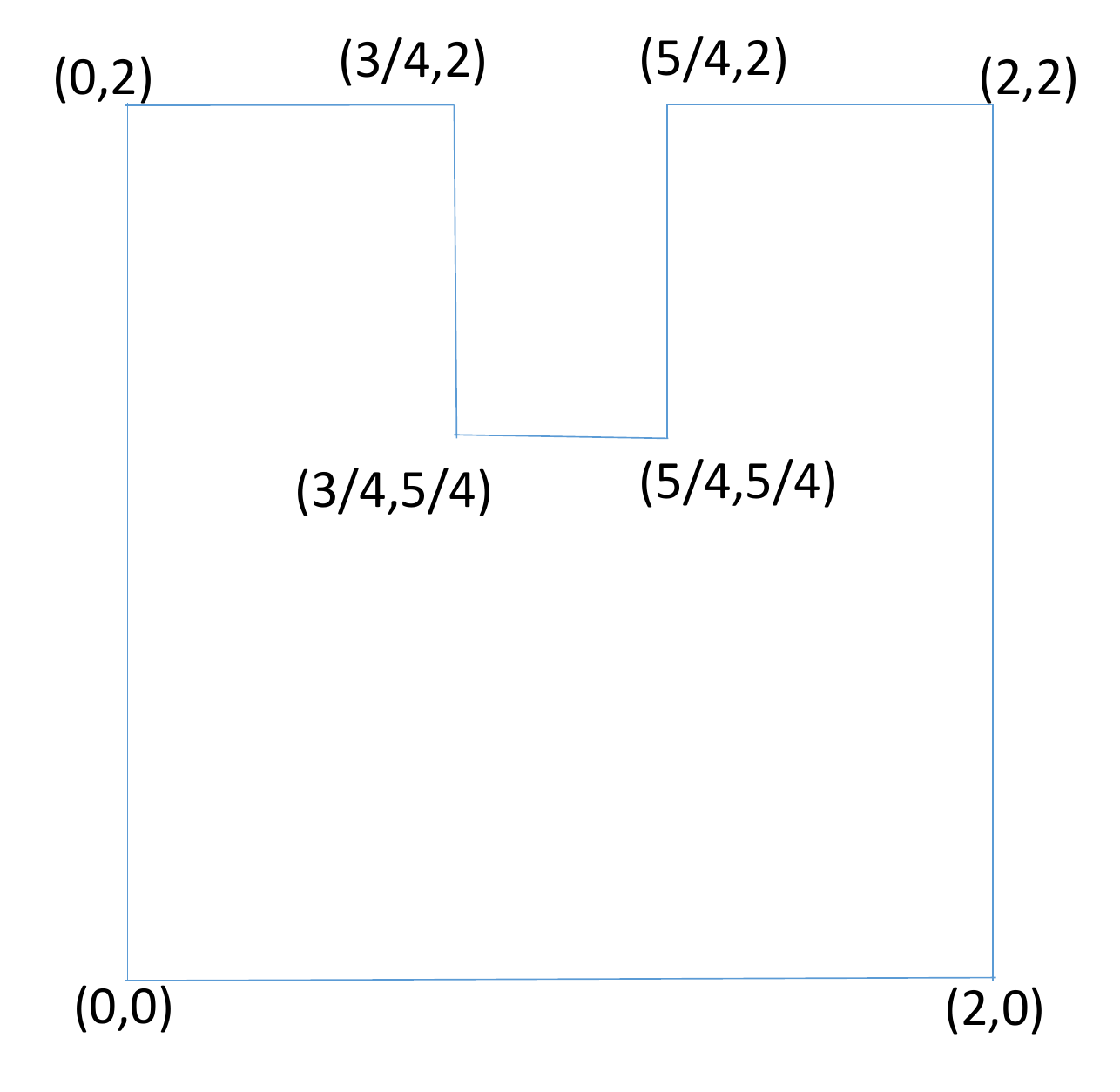}
\vspace{-0.6cm} \caption{Geometry.} \label{geomecorners}
\end{center}
\end{figure}
%

%
%
%
%
We begin by showing comparisons between the uniform and the adaptive methods corresponding to iterative system \eqref{V1hi}. \\
Figures \ref{figurec21}-\ref{figurec24} present the evolution of the mesh during the iterations. We remark that, from an iteration to another, the concentration of the refinement is on the complex vorticity regions, namely at the reentrant corner and some regions of $\Omega$.
\begin{figure}[htbp]
\begin{minipage}[b]{0.450\linewidth}
 \centering
\includegraphics[width=6cm]{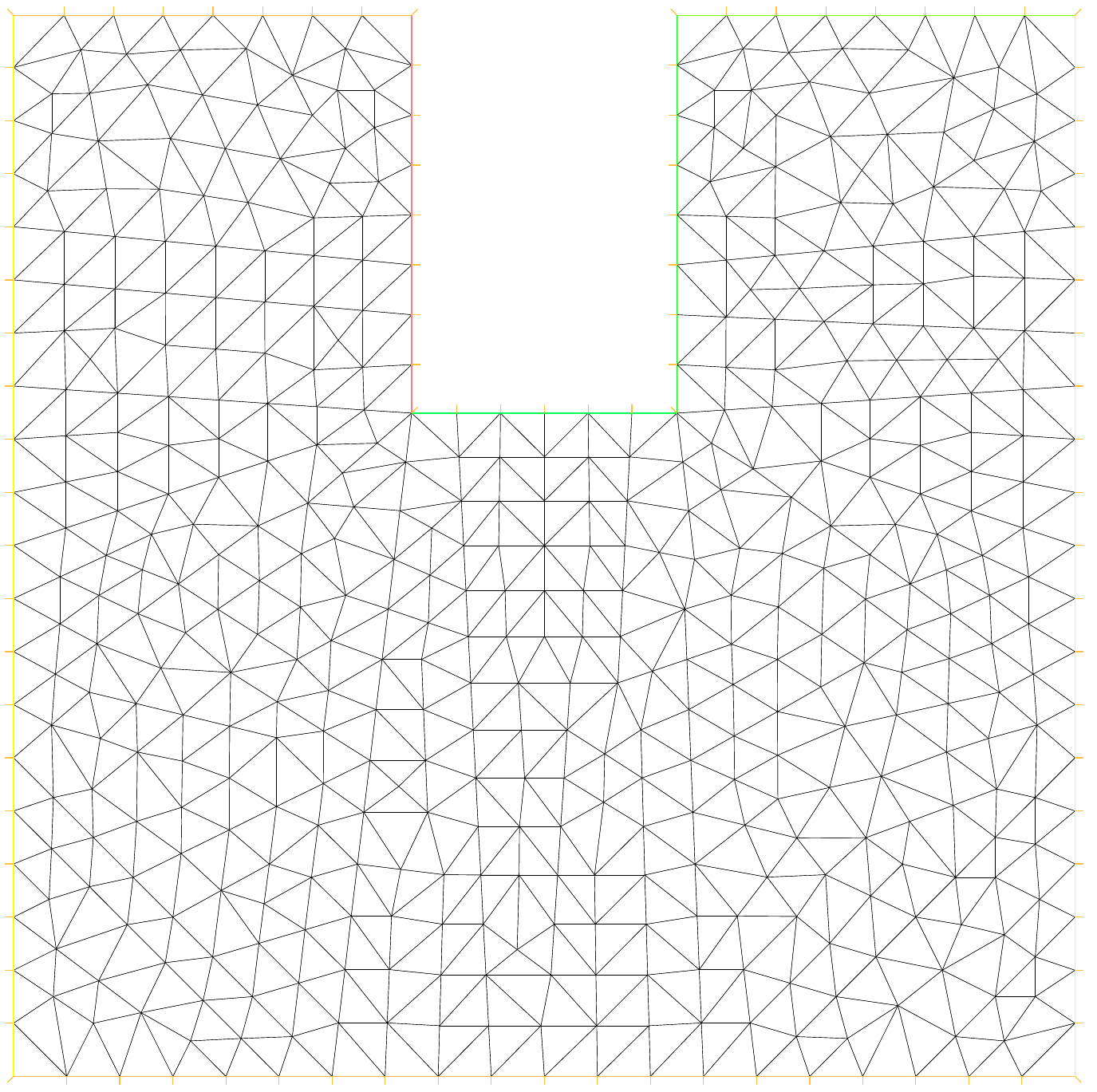}
\vspace{-.4cm}
\caption{Initial mesh (912 triangles)} \label{figurec21}
\end{minipage}\hfill
\begin{minipage}[b]{0.45\linewidth}
\centering
\includegraphics[width=6cm]{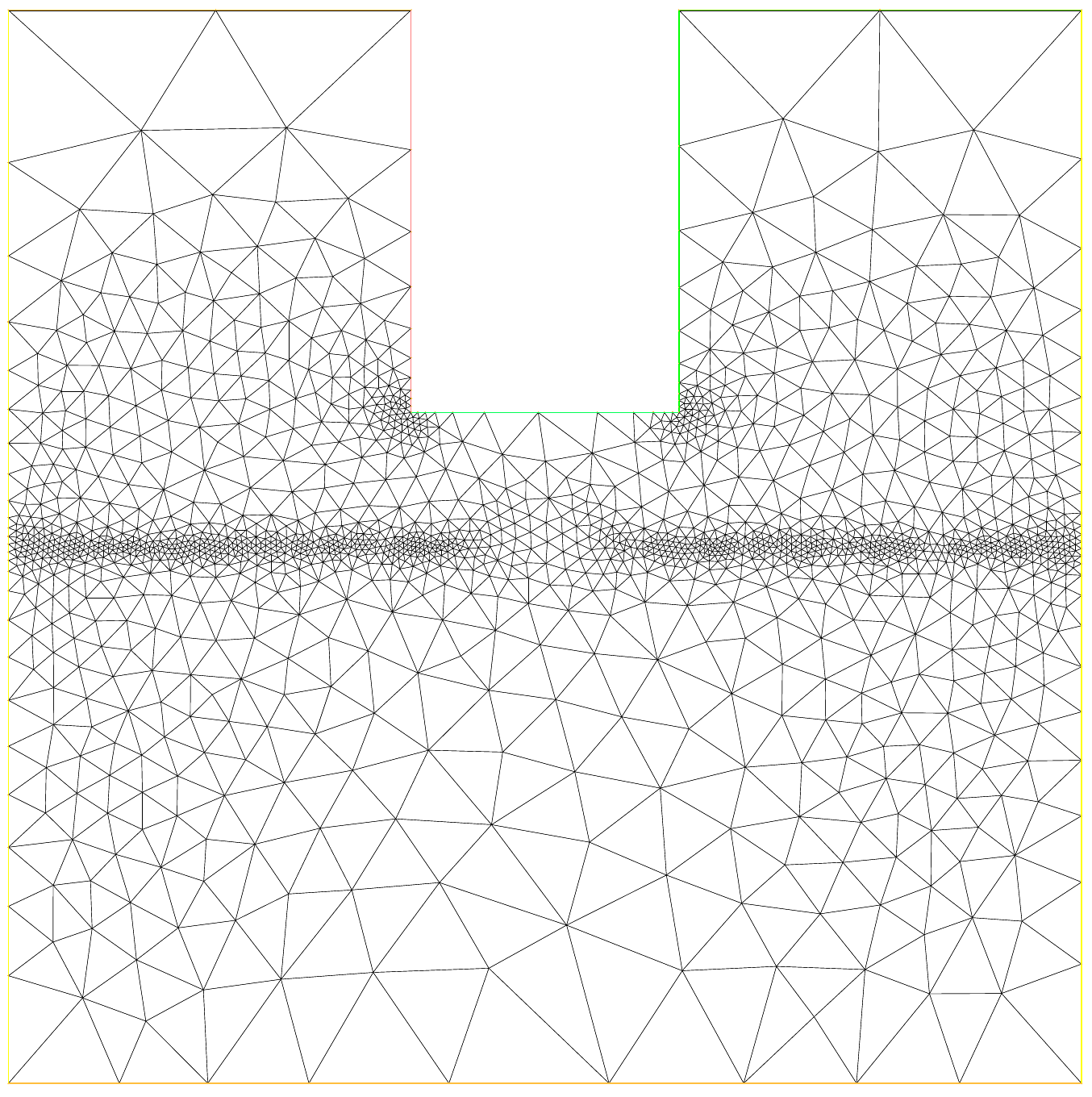}
\vspace{-.4cm}
\caption{ Second level mesh (3956 triangles)} \label{ figurec22 }
\end{minipage}\hfill
\begin{minipage}[b]{0.45\linewidth}
\centering
\includegraphics[width=6cm]{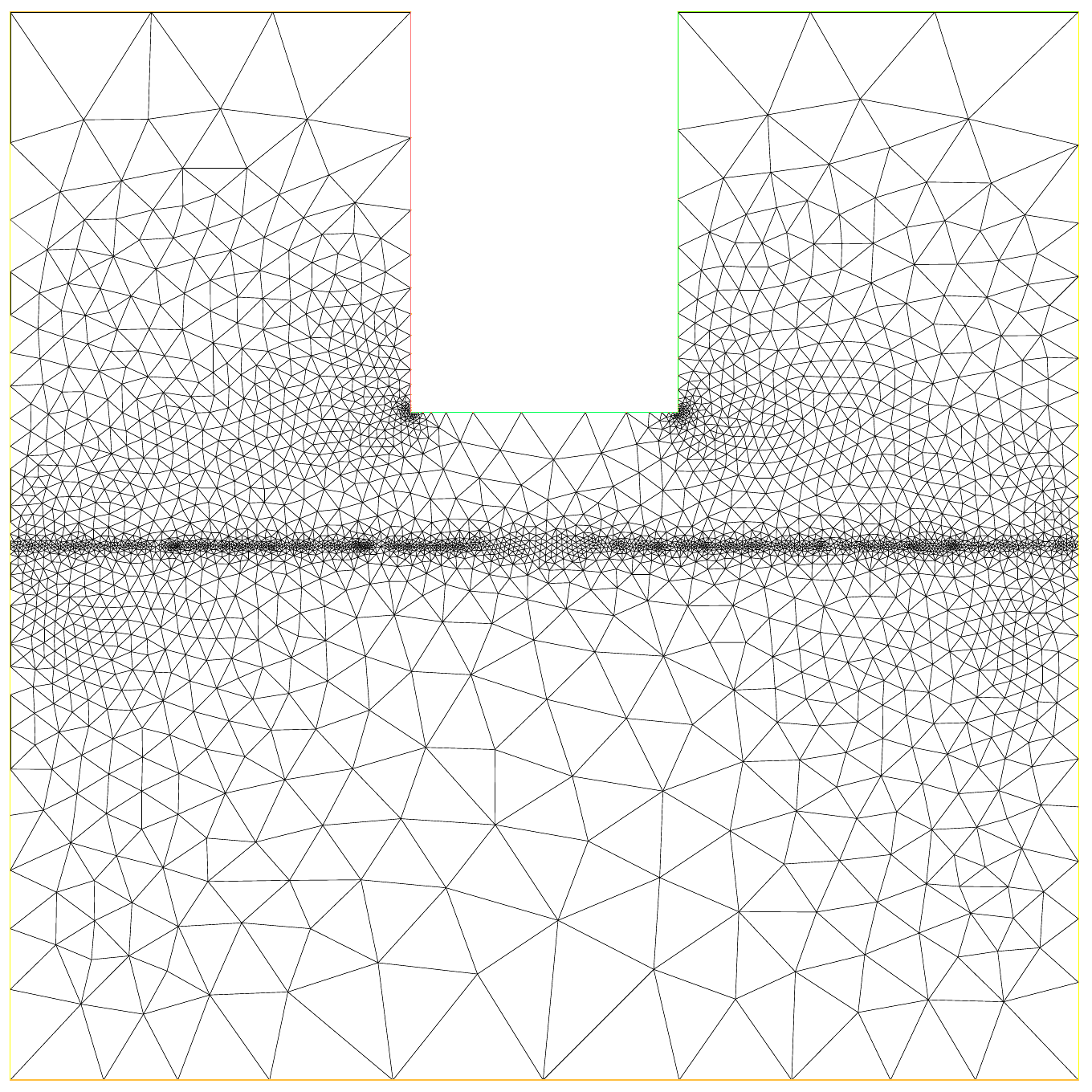}
\vspace{-.4cm}
\caption{ Third level mesh (10389 triangles) } \label{figurec23}
\end{minipage}\hfill
\begin{minipage}[b]{0.45\linewidth}
\centering
\includegraphics[width=6cm]{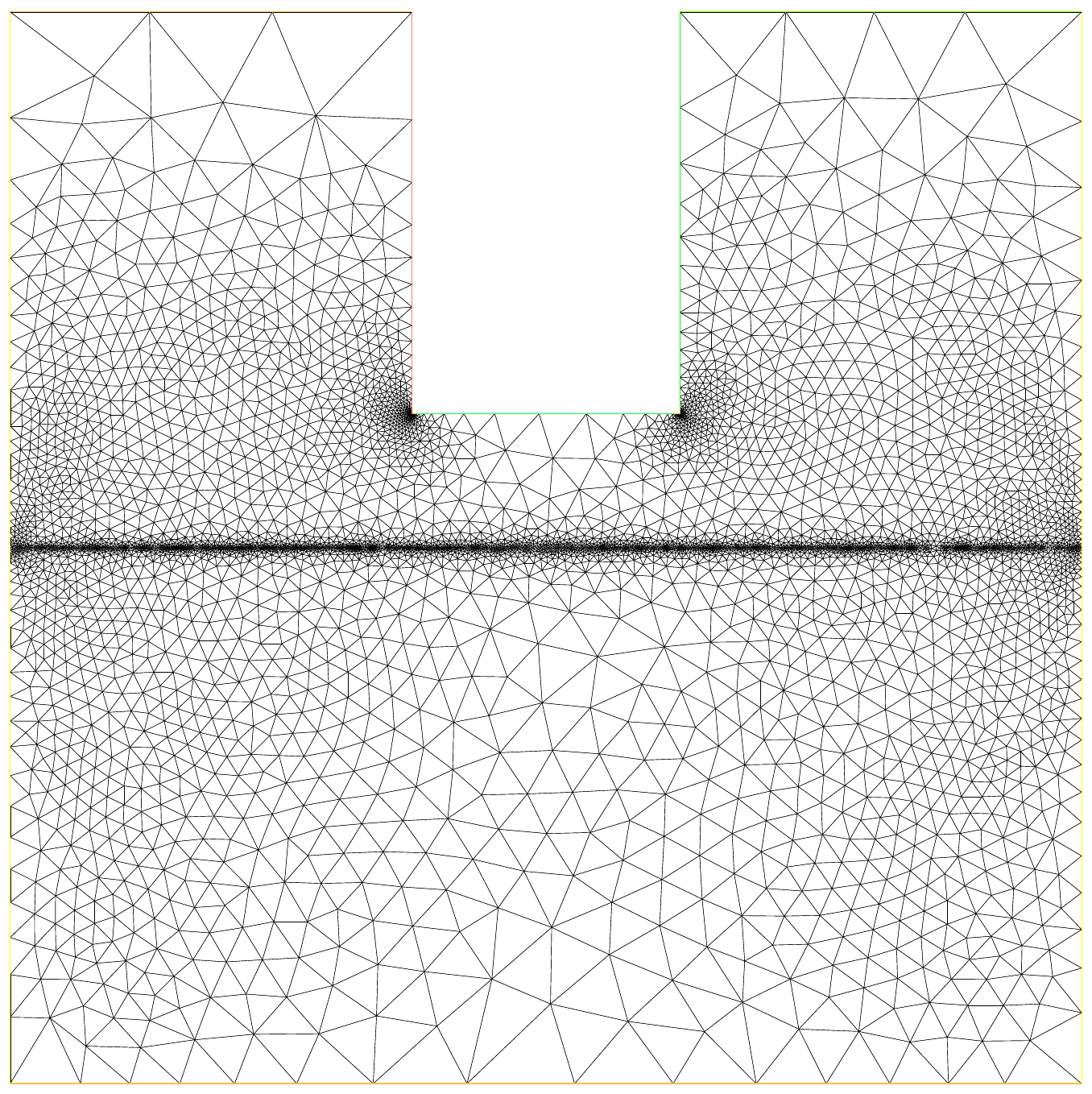}
\vspace{-.4cm}
\caption{ Fourth level mesh (20361 triangles) } \label{figurec24}
\end{minipage}\hfill
\end{figure}

Figures \ref{velocity2_5} and \ref{pressure2_5} show  color  velocity and pressure  at the fourth refinement level. { We} can clearly see that the velocity { in Figure} \ref{velocity2_5} justifies the concentration of the refinement showed in Figure \ref{figurec24}.
\begin{figure}[htbp]
\begin{minipage}[b]{0.45\linewidth}
\hspace{.4cm}
\includegraphics[width=6.6cm]{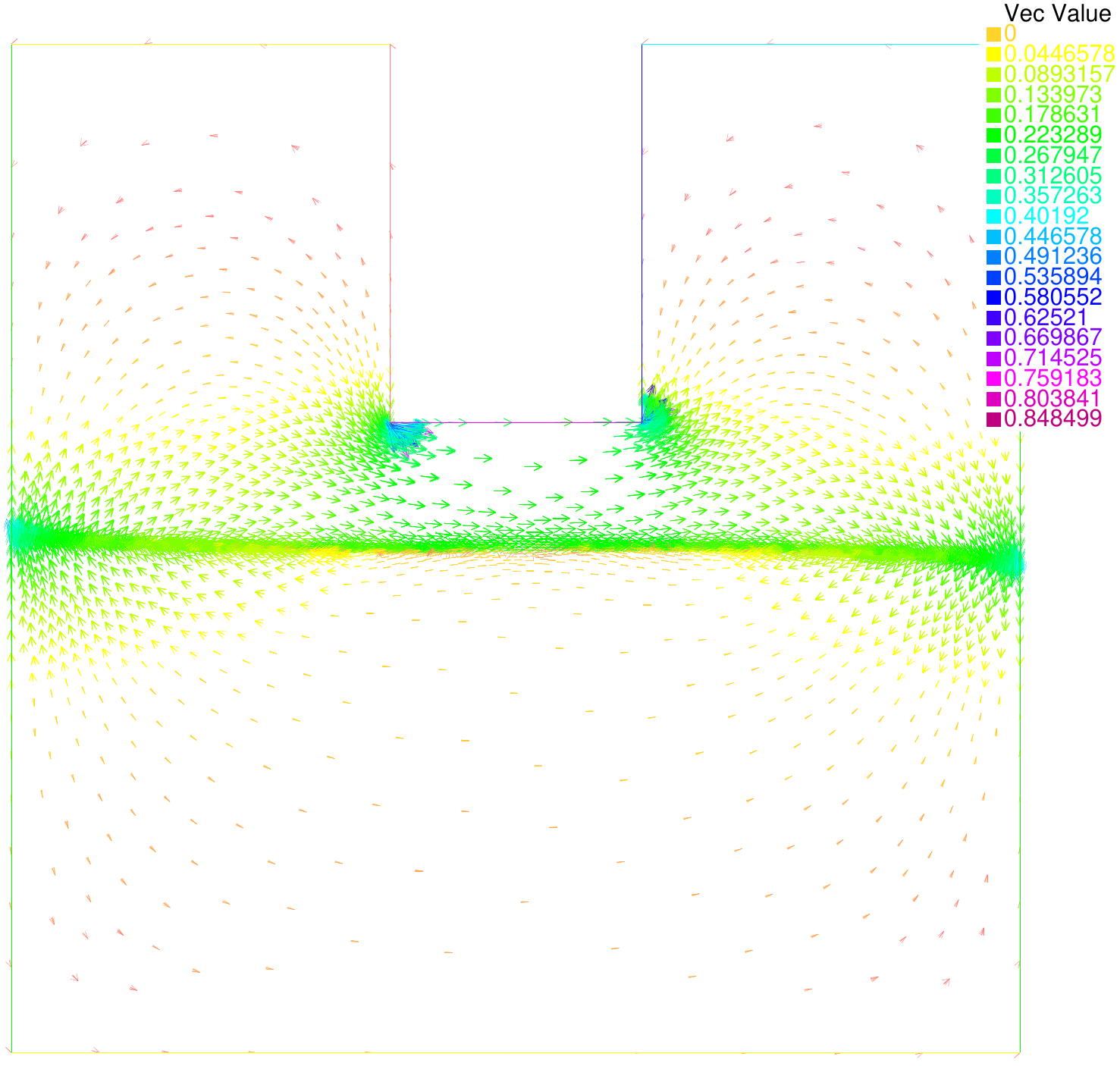}
\vspace{-.4cm}
\caption{numerical velocity at the Fourth refinement level.} \label{velocity2_5}
\end{minipage}\hfill
\begin{minipage}[b]{0.45\linewidth}
\hspace{.4cm}
\includegraphics[width=6.6cm]{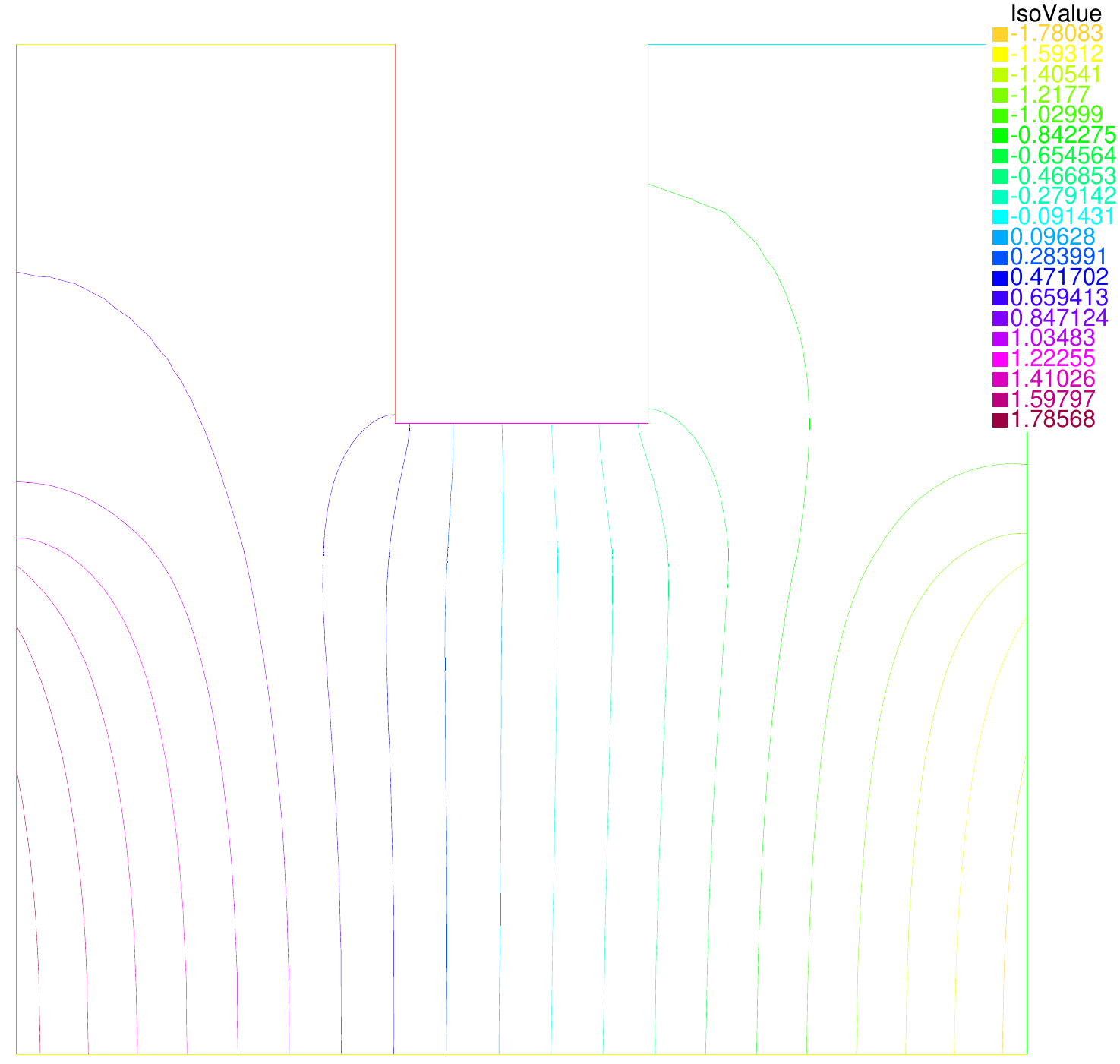}
\vspace{-.4cm}
\caption{ numerical pressure at the Fourth refinement level.} \label{pressure2_5}
\end{minipage}\hfill
\end{figure}

%
%
%
Next, we introduce the relative total error indicator given 
\[
E_{tot} = \ds \frac{\ds   \eta_i^{(D)}}{\ds
 \|\textbf{u}_h^i\|_{L^3(\Omega)} + ||\nabla(p_h^i) ||_{L^{3/2}(\Omega)} }
\]
where here also $\eta_i^{(D)}$ is computed after convergence on the iterations $i$ (by using the stopping criteria \eqref{stoppingg}). 
Figure \ref{figure1.6} shows and compares the relative total error indicator given by $E_{tot}$ between the uniform and adaptive methods. 

\begin{figure}[h!]
\centering
\includegraphics[width=9.5cm]{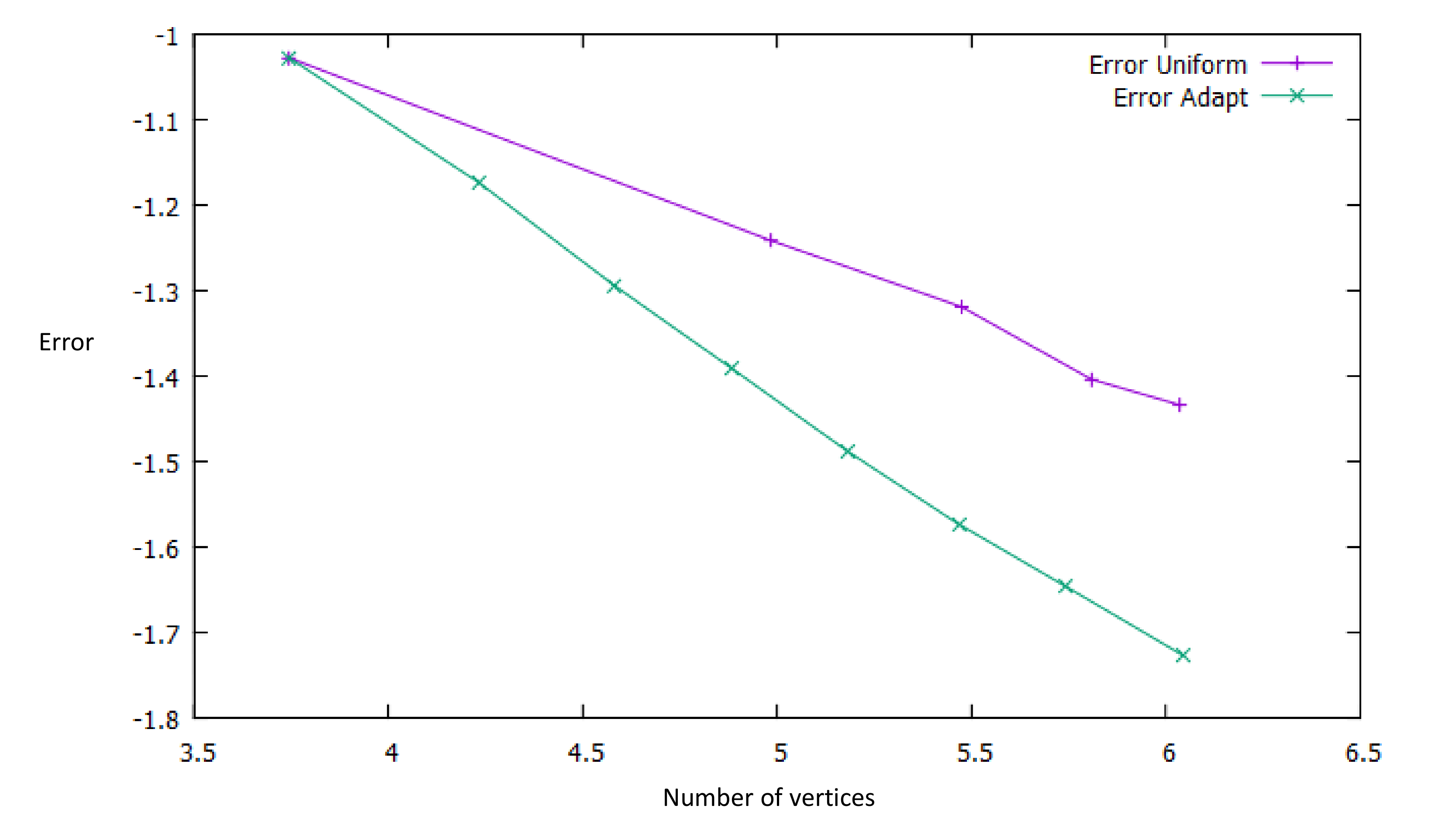}
\vskip -.4cm
\caption{Comparison of the errors $Err$ with respect to the total number of vertices in logarithmic scale for the second case.}\label{figure1.6}
\end{figure}

%

%
%
%
\section{Conclusion}\label{sec:Conclu}
In this \ article, we  discretize a steady Darcy-Forchheimer problem. We  introduce error indicators and establish  optimal {\it a posteriori} error estimates. We perform several numerical simulations where the indicators are used for mesh adaptation, and we show the efficiency of these adaptive methods. \\
%
%
%
%
%
\pagebreak

\end{document}